\documentclass{amsart}
\usepackage{amsfonts}
\usepackage{amsmath,amssymb}
\usepackage{amsthm}
\usepackage{amscd}
\usepackage{graphics}
\usepackage{graphicx}
\theoremstyle{definition}{
\newtheorem{Def}{{\rm Definition}}
\newtheorem{Ex}{{\rm Example}}
\newtheorem{Rem}{{\rm Remark}}

\theoremstyle{plain}
{
\newtheorem{Cor}{Corollary}
\newtheorem{Prop}{Proposition}
\newtheorem{Thm}{Theorem}

\newtheorem{MThm}{Main Theorem}
}

\begin{document}
\title[Special generic maps into ${\mathbb{R}}^k$ where $k \geq5$]{Notes on explicit special generic maps into Euclidean spaces whose dimensions are greater than 4}
\author{Naoki Kitazawa}
\address{Institute of Mathematics for Industry, Kyushu University, 744 Motooka, Nishi-ku Fukuoka 819-0395, Japan\\
 TEL (Office): +81-92-802-4402 \\
 FAX (Office): +81-92-802-4405 \\
}
\email{n-kitazawa@imi.kyushu-u.ac.jp}
\urladdr{https://naokikitazawa.github.io/NaokiKitazawa.html}
\maketitle
\begin{abstract}
{\it Special generic} maps are higher dimensional versions of Morse functions with exactly two singular points on closed manifolds.
They play important roles in so-called Reeb's theorem, characterizing spheres which are not $4$-dimensional topologically and $4$-dimensional unit spheres as smooth manifolds.
The class of such maps also contains canonical projections of unit spheres.

The present paper mainly studies the case where the dimensions of the manifolds of the targets are at least $5$. Special generic maps have been attractive from the viewpoint of algebraic topology and differential topology of manifolds. These maps have been shown to restrict the topologies and the differentiable structures of the manifolds strongly by Calabi, Saeki and Sakuma before 2010s, and later Nishioka, Wrazidlo and the author. So-called {\it exotic} spheres admit no special generic maps in considerable cases and homology groups and cohomology rings are strongly restricted. Such maps into Euclidean spaces whose dimensions are smaller than $5$ have been studied well. 
\end{abstract}
\section{Introduction.}
A {\it special generic} map is a smooth ($C^{\infty}$) map from an $m$-dimensional manifold with no boundary into an $n$-dimensional manifold with no boundary such that at each {\it singular} point it has the form
$(x_1,\cdots x_m) \mapsto (x_1,\cdots,x_{n-1},{\Sigma}_{j=1}^{m-n+1} {x_{n-1+j}}^2)$ for suitable coordinates with $m \geq n \geq 1$. A {\it singular} point $p \in X_1$ of a smooth map $c:X_1 \rightarrow X_2$ is a point at which the rank of the differential ${dc}_p$ is smaller than $\min\{\dim X_1,\dim X_2\}$ where $\dim X$ denotes the dimension of a manifold $X$ (or a space $X$ of a more general class such as the class of polyhedra). 
Morse functions with exactly two singular points are special generic. Reeb's theorem characterizes spheres topologically except $4$-dimensional cases and $4$-dimensional unit spheres as closed and smooth manifolds by such functions. 
%The class of such maps also contains canonical projections of unit spheres.
%【REVISE】 We have deleted the last sentence on canoncail projection of unit spheres. 
 We introduce fundamental terminologies and notation on smooth manifolds.

${\mathbb{R}}^k$ denotes the $k$-dimensional Euclidean space and it is endowed with the standard Euclidean metric: $||x||$ denotes the distance between $x$ and the origin $0$ there.
$S^k:=\{x \in {\mathbb{R}}^{k+1} \mid ||x||=1\}$ ($D^k:=\{x \in {\mathbb{R}}^{k+1} \mid ||x|| \leq 1\}$) denotes the $k$-dimensional unit sphere (resp. disk) for $k \geq 1$.
$\mathbb{R}$ is for ${\mathbb{R}}^1$. $\mathbb{Q} \subset \mathbb{R}$ denotes the ring of all rational numbers and $\mathbb{Z} \subset \mathbb{Q}$ denotes the ring of all integers.

%【REVISE0】 We have added these two paragraphs. These paragraphs are similar to a paragraph before Example \ref{ex:1}. 

A {\it homotopy} sphere is a closed and smooth manifold homotopy equivalent to a sphere: as a result it is homeomorphic to the sphere. A {\it standard} sphere is a homotopy sphere diffeomorphic to a unit sphere. 
%【REVISE】 standard → {\it standard} .
An {\it exotic} sphere is a homotopy sphere which is not diffeomorphic to any standard sphere. $7$-dimensional ones are familiar as ones discovered and studied first. See \cite{milnor} and see also \cite{eellskuiper} for example.

Special generic maps have been shown to restrict the topologies and the differentiable structures of homotopy spheres and the manifolds. 
%【REVISE】 We have added "homotopy spheres and"
For example, canonical projections of unit spheres are special generic whereas exotic spheres admit no such maps in considerable cases.

%【REVISE】 We have deleted "We review known results on manifolds admitting special generic maps." and We have inserted another sentence. 
 Hereafter, the {\it singular set} $S(c)$ of a smooth map $c$ is the set of all singular points of the map. $c(S(c))$ is the {\it singular value set} of $c$ and $Y-c(S(c))$ is the {\it regular value set}. A point in the singular value set is a {\it singular value} and one in the regular value is a {\it regular value}.
\begin{Prop}
\label{prop:1}
For a special generic map $c:X \rightarrow Y$, the singular set is a closed smooth submanifold of dimension $\dim X-1$ of the manifold $c$ and has no boundary. Furthermore, $c {\mid}_{S(c)}$ is a smooth immersion.
\end{Prop}

Note that {\it fold} maps are defined as higher dimensional versions of Morse functions in a similar way. They also satisfy Proposition \ref{prop:1}. Special generic maps are fold maps of course. In our study, we do not concentrate on general fold maps. See \cite{golubitskyguillemin} and \cite{saeki} for general theory of fold maps for example. 
%【REVISE】 the present paper → our study .

 In the present paper, we concentrate mainly on special generic maps into ${\mathbb{R}}^n$ ($n \geq 5$) on closed and simply-connected manifolds whose dimensions are at least $6$.

  Special generic maps on homotopy spheres, special generic maps into ${\mathbb{R}}^n$ ($n \leq 3$) on closed and simply-connected manifolds whose dimensions are at least $4$ and special generic maps into ${\mathbb{R}}^n$ ($n=4$) on $5$-dimensional closed and simply-connected manifolds are studied in \cite{saeki}, followed by \cite{nishioka}. This will be presented in section \ref{sec:2}.

  We present Main Theorems. In our new study, the torsion subgroups of integral homology groups and cohomology rings are key objects. We also explain about some terminologies and notions and the notation on fundamental algebraic topology in section \ref{sec:2}. However, we also assume that we already have elementary knowledge about them. Hereafter, "$\cong$" between two algebraic systems means that they are isomorphic as algebraic systems in the class we consider and "$\oplus$" is for direct sums.
  %【REVISE】 We have added this sentence. 
%【REVISE:Comment】
%In the present paper, we present new examples as the following theorems. 
%【REVISE】 We have deleted this sentence.
\begin{MThm}
%【REVISE0】 We have deleted Theorem 7 and for the present theorem, $D$ is changed to $D_{r}$.
\label{mthm:1}
Let $k_1$ $k_2$, $k$ and $n$ be positive integers satisfying the following relations{\rm :} $k_1<\frac{n}{2}$, $k_1+k_2=n$ and $k_1 \leq k \leq \frac{n}{2}$.
Let $m>n$ be an integer satisfying the relation $k_1,n-k_1-1 \leq m-n+1$.
%【REVISE】 We have changed the exposition on the conditions.
Then there exist families of countably many closed and connected manifolds $\{{M}_{r}\}_{r \in \mathbb{Z}}$ and special generic maps $\{{f}_{r}: M_{r} \rightarrow {\mathbb{R}}^n\}$ and the families enjoy the following three properties.
%【REVISE0】 satisfy → satisfies .
%【REVISE0】 We have deleted "such that the cohomology rings of distinct manifolds are not isomorphic". 
%【REVISEonComment6】 We set the coefficient ring as $\mathbb{Z}$ .
\begin{enumerate}
\item 
\label{mthm:1.1}
For each map $f_r$, the restriction to the singular set is an embedding and the image is diffeomorphic to a manifold obtained from $S^{k_1} \times D^{n-k_1}=S^{k_1} \times D^{k_2}$ by removing the interior of a small closed tubular neighborhood of a smooth closed submanifold in $S^{k_1} \times {\rm Int}\ D^{k_2}$ diffeomorphic to $S^{k_1} \times S^{k-k_1}$ in the case $k>k_1$ and $S^{k_1}$ in the case $k=k_1$.
%【REVISE0】 $S^{k_1} \times {\rm Int}\ D^{k_2}$ → $S^{k_1} \times {\rm Int}\ D^{k_2}$. In the end of this theorem, we have added a sentence related to this instead.
Furthermore, the image $D_{r}$ enjoys the following properties.
\begin{enumerate}
\item \label{mthm:1.1.1}
 $H_{j}(D_{r};\mathbb{Z}) \cong \mathbb{Z}$ for $j=0,n-1$ and $H_{j}(D_{r};\mathbb{Z}) \cong \{0\}$ for $0<j<k_1$.
\item \label{mthm:1.1.2}
 For $j \neq 0,k_1,k-k_1,n-k-1,n-k+k_1-1,k,n-k_1-1,n-1$, $H_{j}(D_{r};\mathbb{Z}) \cong \{0\}$. 
\item \label{mthm:1.1.3}
% Let $k=k_1$. $H_{k_1}(D_{r};\mathbb{Z}) \cong \mathbb{Z} \oplus \mathbb{Z}$ for $n-k-1=k_1=k<\frac{n}{2}$. $H_{k_1}(D_{r};\mathbb{Z}) \cong H_{n-k-1}(D_{r};\mathbb{Z}) \cong \mathbb{Z}$ for $n-k-1 \neq k_1=k<\frac{n}{2}$.
Let $j=k_1,k-k_1,n-k-1,n-k+k_1-1,k,n-k_1-1$. $H_j(D_{r};\mathbb{Z})$ is free. In the case $j=k_1,n-k-1,n-k+k_1-1,n-k_1-1$, it is not the trivial group. In the case $j=k-k_1,k$, satisfying $j \notin \{k_1,n-k-1,n-k+k_1-1,n-k_1-1\}$, it is the trivial group. 
\item \label{mthm:1.1.4}
$D_{r}$ is simply-connected if $k_1 >1$.
\end{enumerate}
Furthermore, if $m$ is sufficiently large, then we can know the homology group of $M_{r}$ completely from $D_{r}$ applying Poincar\'e duality theorem together with Proposition \ref{prop:3} {\rm (}\ref{prop:3.5}{\rm )}, saying that $H_j(M_{r};\mathbb{Z}) \cong H_j(D_{r};\mathbb{Z})$ holds for $0 \leq j \leq m-n$ and we can also know the homotopy groups ${\pi}_j(M) \cong {\pi}_j(D_r)$. 
%【REVISE】 We have changed our exposition on Proposition \ref{prop:3}.
\item
\label{mthm:1.2} If $r \neq 0$, then 
 $M_{r}$ admits no special generic maps into ${\mathbb{R}}^{n^{\prime}}$ for $1 \leq n^{\prime} <n$. 
 \item 
 \label{mthm:1.3}
 $M_r$ also admits a special generic map into ${\mathbb{R}}^{n^{\prime}}$ for $n \leq n^{\prime} \leq m$. 
 \item 
 \label{mthm:1.4}
 We can have the family $\{M_{r}\}_{r \in \mathbb{Z}}$ in such a way that
the integral cohomology rings of distinct manifolds $M_{r_1}$ and $M_{r_2}$ are not isomorphic for distinct integers $r_1, r_2 \in \mathbb{Z}$ if the relation $k_1,n-k_1-1 <m-n+1$ is satisfied.
\end{enumerate}

\end{MThm}
\begin{MThm}[Theorem \ref{thm:8}]
	%【REVISE0】 Theorem \ref{thm:8} and so on → E. g. Theorem \ref{thm:7} .
\label{mthm:2}
There exists a closed and simply-connected manifold of dimension $m>6$ admitting a special generic map into
${\mathbb{R}}^{n}$ for $5 \leq n \leq m$ and admitting no special generic maps into ${\mathbb{R}}^{n}$ for $1 \leq n \leq 4$.
\end{MThm}
%【REVISE】 We have added "there exist".
Hereafter, (boundary) connected sums of (smooth) manifolds are considered in the smooth category unless otherwise stated. The {\it Stiefel-Whitney classes} and the {\it Pontrjagin classes} of a smooth manifold will be presented shortly in the third section. However, we also assume fundamental knowledge on them. See also \cite{milnorstasheff}.
%【REVISE】 We have added this sentenc.  
\begin{MThm}
		%【REVISE0】 We have deleted "." after \ref{thm:9}.
\label{mthm:3}
Let $l>0$ be an integer and $\{G_j\}_{j=0}^9$ a sequence of finitely generated commutative groups of length $10$ satisfying the following conditions.
\begin{itemize}
\item $G_j$ and $G_{m-j}$ are isomorphic for $0 \leq j \leq 9$ and $G_0$ is isomorphic to $\mathbb{Z}$ and $G_1$ is the trivial group. For $j \neq 2,6$, $G_j$ is free. The ranks of $G_4$ and $G_5$ at least $2l$.
\item The torsion subgroup of $G_2$ is isomorphic to a group represented as a direct sum of $2l$ copies of the group of order $2$.
\end{itemize}
%【REVISE0】 We use "itemize" .
Then, there exists a pair $(M_1,M_2)$ of $9$-dimensional closed and simply-connected manifolds enjoying the following properties.
\begin{enumerate}
\item \label{mthm:3.1}
 The $j$-th integral homology groups of $M_1$ and $M_2$ are isomorphic to $G_j$ for $0 \leq j \leq 9$.
\item \label{mthm:3.2}
The sets consisting of all elements whose orders are infinite and the zero elements of the integral cohomology rings of $M_1$ and $M_2$ are subalgebras isomorphic to the integral cohomology ring of a $9$-dimensional manifold, which is represented as a connected sum of products of standard spheres and presented in Example \ref{ex:1}.
\item \label{mthm:3.3}
Let $i=1,2$. $M_i$ admits a special generic map into ${\mathbb{R}}^{n}$ for $i+4 \leq n \leq 9$ whereas it admits no special generic maps into ${\mathbb{R}}^{n}$ for $1 \leq n \leq i+3$. Furthermore, we can construct the special generic map on $M_i$ so that the restriction to the singular set is an embedding.
\item \label{mthm:3.4}
 The $j$-th Stiefel-Whitney classes and the $j$-th Pontrjagin classes of $M_1$ and $M_2$ are the zero elements for any positive integer $j$. 

\end{enumerate}
\end{MThm}
%【REVISE】 "Example 2" in the previous version has been deleted. v

The content of the present paper is as follows. In the next section, we review the structures of special generic maps and fundamental construction.  
The third section is devoted to expositions of Main Theorems including proofs. We also present Main Theorem \ref{mthm:4}, which is a variant of Main Theorem \ref{mthm:3} for the $8$-dimensional case as another main theorem.
%【REVISE】 We have deleted "the review of Theorem \ref{thm:5} and" . 
\section{Structures of special generic maps.}
\label{sec:2}
In the present paper, a diffeomorphism on a manifold is assumed to be smooth and the {\it diffeomorphism group} of
the manifold is the group consisting of all diffeomorphisms on the manifold. This group is endowed with the so-called {\it Whitney $C^{\infty}$ topology}. A {\it smooth} bundle is a bundle whose fiber is a smooth manifold and whose structure group is the diffeomorphism group.
%【REVISE】 We have moved these two sentences here.

A {\it linear} bundle is a smooth bundle whose fiber is a unit disk (sphere) and whose structure group acts on the fiber as linear transformations.
%【REVISE0】 disc → disk .
\begin{Prop}[\cite{saeki2}]
\label{prop:2}
\begin{enumerate}
\item
\label{prop:2.1}
Let $f:M \rightarrow N$ be a special generic map from an $m$-dimensional closed manifold into an $n$-dimensional manifold with no boundary satisfying $m>n \geq 1$. We have an $n$-dimensional compact smooth manifold $W_f$ and $f$ is the composition of a smooth map $q_f:M \rightarrow W_f$ with a smooth immersion $\bar{f}:W_f \rightarrow N$.
Furthermore, for some small collar neighborhood $N(\partial W_f)$, the composition of the restriction to the preimage of the collar neighborhood with the canonical projection to the boundary gives a linear bundle whose fiber is the {\rm (}$m-n+1${\rm )}-dimensional unit disk $D^{m-n+1}$ and on the preimage of the complementary set of the interior of the collar neighborhood, $f$ gives a smooth bundle whose fiber is diffeomorphic to the unit sphere $S^{m-n}$. In particular, for example, it is a linear bundle whose fiber is the unit sphere $S^{m-n}$ for $m-n=1,2,3$.
%【REVISE0】 disc → disk .
\item
\label{prop:2.2}
For a smooth immersion ${\bar{f}}_N$ of an $n$-dimensional compact manifold $\bar{N}$ into an $n$-dimensional manifold $N$ with no boundary and for any integer $m>n$, there exist an $m$-dimensional closed manifold $M$ and a special generic map $f$ into $N$ such that $W_f=\bar{N}$ and $\bar{f}={\bar{f}}_N$ hold in the previous statement and that the bundle over $\partial N(\partial W_f)$ and the bundle over the complementary set $W_f-{\rm Int}\ N(\partial W_f)$ of the interior of $N(\partial W_f) \subset W_f$ in the previous statement are trivial. If $N$ is orientable {\rm (}connected{\rm )}, then we can construct $M$ as an orientable {\rm (}resp. a connected{\rm )} manifold. If $N$ is connected and orientable, then we can obtain $M$ satisfying both.
%【REVISE0】　{\rm (}a connected{\rm )} → ｛\rm }resp. a connected{\rm )} .

\end{enumerate}
\end{Prop}

\begin{Ex}
	\label{ex:1}
	Canonical projections of unit spheres are special generic maps whose singular sets are equators and standard spheres and whose restrictions to the singular sets are embeddings.
	Let $l,m,n>0$ be integers satisfying $m \geq n$. A manifold represented as a connected sum of $l$ manifolds each of which is diffeomorphic to the corresponding manifold in $\{S^{k_j} \times S^{m-k_j}\}_{j=1}^l$ satisfying $1 \leq k_j <n$ admits a special generic map into ${\mathbb{R}}^n$ satisfying the following three.
	\begin{enumerate}
		\item The restriction to the singular set is an embedding.
		\item The image is a manifold represented as a boundary connected sum of $l$ manifolds each of which is diffeomorphic to the corresponding manifold in the family $\{S^{k_j} \times D^{n-k_j}\}_{j=1}^l$.
		\item This map is as in Proposition \ref{prop:2} (\ref{prop:2.2}).
	\end{enumerate}
\end{Ex}

\begin{Thm}[\cite{calabi, saeki2, saeki3, wrazidlo}]
	%【REVISE】 \cite{calabi}, \cite{saeki2}, \cite{saeki3} and \cite{wrazidlo} → \cite{calabi, saeki2, saeki3, wrazidlo} .
	\label{thm:1}
	An exotic sphere of dimension $m>4$ admits no special generic maps into $\mathbb{R}^n$ for $n=m-3,m-2,m-1$. Homotopy spheres except $4$-dimensional exotic spheres, which are undiscovered, admit special generic maps into the plane.
	$7$-dimensional oriented homotopy spheres of at least 14 types of 28 types admit no special generic map into
	${\mathbb{R}}^3$.
\end{Thm}

\begin{Thm}[\cite{saeki2,saekisakuma,saekisakuma2}]
	%【REVISE0】 \cite{saeki2}, \cite{saekisakuma} and \cite{saekisakuma2}. → \cite{saeki2,saekisakuma,saekisakuma2}
	\label{thm:2}
	Let $m=3,4,5,6,7$.
	An $m$-dimensional closed and connected manifold admits a special generic map into the plane if and only if it is represented as a connected sum of the total spaces of smooth bundles over $S^1$ whose fibers are diffeomorphic to the unit sphere $S^{m-1}$.
	
	A $4$-dimensional closed and connected manifold whose fundamental group is free admits a special generic map into ${\mathbb{R}}^3$ if and only if it is a standard sphere or represented as a connected sum of manifolds satisfying either of the following two.
	%【REVISE0】 We havce inserted "a standard sphere or" before "represented".
	Furthermore, there exist infinitely many $4$-dimensional closed and connected manifolds which are homeomorphic to these manifolds, admit fold maps into ${\mathbb{R}}^3$, and admit no special generic maps into ${\mathbb{R}}^3$.
	\begin{enumerate}
		\item The total space of a linear bundle over $S^1$ whose fiber is diffeomorphic to the unit sphere $S^3$.
		\item The total space of a linear bundle over $S^2$ whose fiber is diffeomorphic to the unit sphere $S^2$.
	\end{enumerate}
	%【REVISE0】 A total space → The total space .
	A $5$-dimensional closed and simply-connected manifold admits a special generic map into ${\mathbb{R}}^3$ if and only if it is a standard sphere or represented as a connected sum of the total spaces of linear bundles over $S^2$ whose fibers are diffeomorphic to the unit sphere $S^3$.
	%【REVISE0】 We havce inserted "a standard sphere or" before "represented".
\end{Thm}
%In the case where the manifold of the target is the plane, closed manifolds admitting special generic maps are completely classified generally in \cite{saeki2}.
%【REVISE0】 the target → the manifold of the target .
% More precisely, these manifolds whose dimensions are not $5$ are characterized as ones represented as connected sums of total spaces of smooth bundles over $S^1$ whose fibers are diffeomorphic to homotopy spheres for example.

In the case where the manifold of the target is ${\mathbb{R}}^3$, more facts are shown in \cite{saeki2, saekisakuma}, for example.
%【REVISE0】 the target → the manifold of the target .
%【REVISE0】 \cite{saeki2}, \cite{saekisakuma}, and so on → \cite{saeki2, saekisakuma}, for example
%【REVISE0】 We have removed the paragraphs here.

Hereafter we introduce fundamental notions and the notation on algebraic topology. \cite{hatcher} gives a systematic exposition.

For a pair of two topological spaces $X$ and $X^{\prime} \subset X$ and a commutative group $A$, the $j$-th (co)homology group of the pair is denoted by $H_j(X,X^{\prime};A)$ (resp. $H^{j}(X,X^{\prime};A)$) where $A$ is the coefficient ring.
%【REVISE】 ring $A$ → group $A$. 
%【REVISE】 (X,X;A) → (X,X^{\prime};A) .
 If $X^{\prime}$ is the empty set, then this is the $j$-th homology (resp. cohomology) group of $X$ and denoted by $H_j(X;A)$ (resp. $H^j(X;A)$).
 
 For a topological space $X$, ${\pi}_j(X)$ denotes the $j$-th homotopy group. 
 
 Let $A$ be a commutative ring. Let $c:(X_1,{X_1}^{\prime})\rightarrow (X_2,{X_2}^{\prime})$ be a continuous map between pairs of topological spaces enjoying the relation ${X_i}^{\prime} \subset X_i$ for $i=1,2$ which enjoys the relation $c({X_1}^{\prime}) \subset {X_2}^{\prime}$. We introduce the notation of canonically defined homomorphisms as $c_{\ast}:H_j(X_1,{X_1}^{\prime};A) \rightarrow H_j(X_2,{X_2}^{\prime};A)$ and $c^{\ast}:H^j(X_2,{X_2}^{\prime};A) \rightarrow H^j(X_1,{X_1}^{\prime};A)$. Let ${X_1}^{\prime}={X_2}^{\prime}$ be the empty sets. Let $c_{\ast}:{\pi}_j(X_1) \rightarrow {\pi}_j(X_2)$ denote the homomorphism. 

For a topological space $X$ and a commutative ring $A$, $H^{\ast}(X;A)$ denotes the direct sum ${\oplus}_{j=1}^{\infty} H^j(X;A)$ of all $j$-th cohomology groups for $j \geq 0$. This is also the cohomology ring of $X$ with the coefficient ring $A$. This is the graded algebra where the product of the sequence $\{u_j\}_{j=1}^l \subset H^{\ast}(X;A)$ of $l>0$ elements is the {\it cup product} ${\cup}_{j=1}^l u_j$. $u_1 \cup u_2$ is for the case $l=2$.

For a pair of two topological spaces $X$ and $X^{\prime} \subset X$ and a commutative ring $A$, we can also define the {\it cup product} $u_1 \cup u_2$ for $u_1 \in H^{j_1}(X;A)$ and $u_2 \in H^{j_2}(X,X^{\prime};A)$ where $j_1$ and $j_2$ are arbitrary (non-negative) integers. 
The former cup products are for specific cases.
For a compact, connected and orientable manifold $X$, we can define the {\it bilinear form} for $X$ as the bilinear form mapping $(u_1,u_2) \in H^{j_1}(X;A) \times H^{j_2}(X,X^{\prime};A)$ to $u_1 \cup u_2 \in H^{j_1+j_2}(X,X^{\prime};A)$ where $(j_1,j_2)$ is an arbitrary pair of (non-negative) integers. This is important in Theorem \ref{thm:7} and Poincar\'e duality theorem for a compact, connected and orientable manifold $X$ whose boundary $\partial X$ is not empty for example. We also need no orientations for the manifolds and Poincar\'e duality theorem for them in the case where the coefficient ring $A$ is $\mathbb{Z}/2\mathbb{Z}$, the field of order $2$, for example.

If $A=\mathbb{Z}$, then such groups and graded algebras are said to be integral homology groups, cohomology groups and cohomology rings. If $A=\mathbb{Q}$, then such groups and graded algebras are said to be rational homology groups, cohomology groups and cohomology rings.   

We also explain about Poincar\'e duality again and the {\it cohomology duals} to some elements of homology groups in the third section.

%【REVISE0】 We introduce elementary algebraic topology.
%【REVISE】 For example, we have added expositions on Poincar\'e duality and the duals to elements of a basis of a submodule of a homology group.

\begin{Thm}[E.g. \cite{nishioka}]
	\label{thm:3}
	If a closed and simply-connected manifold of dimension $m>4$ admits a special generic map into ${\mathbb{R}}^4$, then its $2$nd integral homology group is free. Especially, a $5$-dimensional closed and connected manifold admits a special generic map into ${\mathbb{R}}^4$ if and only if it is represented as a connected sum of the total spaces of linear bundles over $S^2$ whose fibers are diffeomorphic to the unit sphere $S^3$.
\end{Thm}
%【REVISE】 sum of total → some of the total . 
%\cite{kitazawa11,kitazawa12} are important in generalizing the dimension $m$.
%【REVISE0】 We have added an exposition on \cite{kitazawa11,kitazawa12} where "\cite{kitazawa11,kitazawa12}" in Theorem \ref{thm:3} are removed.

An ({\it $m$-dimensional}) {\it rational homology sphere} is a closed manifold whose $j$-th rational homology group is isomorphic to that of $S^m$ for any integer $j$.
\begin{Thm}[\cite{wrazidlo2}]
	\label{thm:4}
	Let $k>2$ and $1 \leq n < 2k+1$ be integers. If a {\rm (}$2k+1${\rm )}-dimensional rational homology sphere admits a special generic map into ${\mathbb{R}}^n$, then the $k$-th integral homology group must be finite and the order is the square of some integer.
\end{Thm}
The following theorem is a key ingredient in the proof of Main Theorems of the present paper.
%【REVISE】 We prove Theorem 5 here.
\begin{Thm}[\cite{kitazawa14}]
	\label{thm:5}
	Let $f:M \rightarrow N$ be a special generic map from an $m$-dimensional closed manifold into an $n$-dimensional connected non-closed manifold $N$ with no boundary satisfying $m>n \geq 1$.
		%【REVISE】 non-closed manifold → non-closed manifold $N$ with no boundary .
	 Let $A$ be a commutative ring.

	 Assume that there exists a sequence $\{a_j \in H^{\ast}(M;A)\}_{j=1}^l$ whose degrees are at most $m-n$ of length $l>0$ and that the sum of the degrees of these $l$ elements is at least $n$. Then, the cup product ${\cup}_{j=1}^l a_j$ is the zero element. 
	%【REVISE0】 connected open → connected non-closed .
	%【REVISE0】 the cup product ${\prod}_{j=1}^l a_j$ → the cup product ${\cup}_{j=1}^l a_j$ .
\end{Thm}
\begin{Prop}[\cite{kitazawa1, kitazawa2, kitazawa3, kitazawa4, saeki, saeki2, saekisuzuoka}]
	\label{prop:3}
	In the situation of Proposition \ref{prop:2}, we can construct an {\rm (}$m+1${\rm )}-dimensional compact manifold $W$                                                                                                                                                                                                                              in the topology category {\rm (}PL category, or equivalently, in the piecewise smooth category,{\rm )} enjoying the following properties.
	%【REVISE】 in the PL category, or equivalently, in the piecewise smooth category, → in the topology category and the PL category, or equivalently, in the piecewise smooth category,} .
    %【REVISE】 We changed expositions ro ones respecting the category.
	\begin{enumerate}
		\item \label{prop:3.1}
		There exists a continuous {\rm (}resp. PL or piecewise smooth{\rm )} map $r:W \rightarrow W_f$.
		\item \label{prop:3.2}
		$r$ is a deformation retract and in the PL or piecewise smooth category, this gives a collapsing of 
		$W$ to $W_f$.
		\item \label{prop:3.3}
		 $M$ is the boundary $\partial W$ of $W$ where in the case of the PL or piecewise smooth category we discuss in the PL or piecewise smooth category regarding $M$ as a PL or piecewise smooth manifold in a canonical way. 
		\item \label{prop:3.4}
		 $q_f=r {\mid}_{\partial W}=r {\mid}_M$.
		\item \label{prop:3.5}
		 For the inclusion $i:M = \partial W \rightarrow W$, $i_{\ast}:H_j(M;A) \rightarrow H_j(W;A)$, $i_{\ast}:{\pi}_j(M;A) \rightarrow {\pi}_j(W;A)$ and $i^{\ast}:H^j(W;A) \rightarrow H^j(M;A)$ are isomorphisms for any integer $0 \leq j \leq m-n$ and any commutative group $A$.
	\end{enumerate}
	
	Furthermore, for example in the case $m-n=1,2,3$, we can always discuss this in the smooth category.
	Furthermore, we can construct a special generic map in Proposition \ref{prop:2} {\rm (}\ref{prop:2.2}{\rm )} so that we can discuss this in the smooth category.
\end{Prop}
We can give a proof of Theorem \ref{thm:5} by applying Proposition \ref{prop:3}.
\begin{proof}[A proof of Theorem \ref{thm:5}.]
	We can take a unique element $b_j \in H^{\ast}(W;A)$ satisfying $a_j=i^{\ast}(b_j)$ for the inclusion $i:M \rightarrow W$ in Proposition \ref{prop:3}.
	We have ${\cup}_{j=1}^l a_j=i^{\ast}({\cup}_{j=1}^l b_j)$ and this is the zero element since the degree is greater than or equal to $n$, which is the dimension of $W_f$. 
	%【REVISE0】 \prod → \cup .
	Remember that $W$ is (simple) homotopy equivalent to $W_f$ and that its boundary is non-empty boundary. $W$ has the homotopy type of a polyhedron whose dimension is smaller than $n$. 
	%【REVISE】　has → that its boundary is .
	For this, note also that $N$ is non-closed and has no boundary.
	%【REVISE】　We have added a sentence.
	This completes the proof.
	
\end{proof}
\begin{Thm}[\cite{kitazawa13}]
	\label{thm:6}
	There exist countably many $7$-dimensional closed and simply-connected manifolds the cohomology ring of each of which is isomorphic to that of a manifold in Example \ref{ex:1} for any coefficient ring which is a principal ideal domain having an identity element different from the zero element and these manifolds enjoy the following two properties.
	\begin{enumerate}
		\item These manifolds admit fold maps into ${\mathbb{R}}^n$ for any $n=1,2,3,4,5,6,7$.
		\item These manifolds admit no special generic maps into ${\mathbb{R}}^n$ for any $n=1,2,3,4,5$.
	\end{enumerate} 
\end{Thm}
%【REVISE】 infinitely → countably
%\begin{Rem}
%	The author introduced and systematically studied a class of fold maps: {\it round} fold maps. This class contains canonical projections of unit spheres as simplest examples. The author found that all $7$-dimensional homotopy spheres admit round fold maps into ${\mathbb{R}}^4$ satisfying additional conditions on singular sets, singular value sets, preimages and structures and found that the differentiable structures and topological topological properties of the maps are closely related. See \cite{kitazawa, kitazawa2, kitazawa3, kitazawa4}.
%【REVISE】 ,structures and so on, → and structures .	
%\end{Rem}
%【REVISE】 We have moved much of Introduction.
%【REVISE】 We have deleted Remark 1.
%【REVISE】 We have deleted Example 2.
\begin{Thm}[\cite{kitazawa15}]
	\label{thm:7}
 In Proposition \ref{prop:2} {\rm (}\ref{prop:2.2}{\rm )}, assume that $\bar{N}=W_f$ is connected and orientable. Then we can also have a suitable closed and connected manifold $M$ and a special generic map $f:M \rightarrow N$ enjoying the following properties.
	\begin{enumerate}
		\item For any principal ideal domain $A$ having the unique identity element different from the zero element and making $H_j(W_f;A)$ free for any integer $j$, the following properties hold.
 \begin{enumerate}
 	\item $H_j(M;A)$ is isomorphic to a direct sum of $H_j(W_f;A)$ and $H_{j-(m-n)}(W_f,\partial W_f;A)$.
  \item $H^j(M;A)$ is isomorphic to a direct sum of $H^j(W_f;A)$ and $H^{j-(m-n)}(W_f,\partial W_f;A)$ and we can identify them as commutative groups and modules over $A$.
  \item For the structure of the cohomology ring $H^{\ast}(M;A)$, the following three hold.
  \begin{enumerate}
  	\item The cup products of elements in the summand $H^{\ast}(W_f;A)$ induce a subalgebra isomorphic to the cohomology ring $H^{\ast}(W_f;A)$.
  	\item The cup products of elements in the summand $H^{\ast}(W_f;A)$ and elements in ${\oplus}_j H^{j-(m-n)}(W_f,\partial W_f;A)$ induce a bilinear form isomorphic to that for $W_f${\rm :} ${\oplus}_j H^{j-(m-n)}(W_f,\partial W_f;A)$ denotes the direct sum for all integers $j \geq m-n$.
  	\item The cup products of elements in the summand ${\oplus}_j H^{j-(m-n)}(W_f,\partial W_f;A)$ are always the zero elements.
  \end{enumerate}
\end{enumerate}
\item We have a special generic map $f_j:M \rightarrow N \times {\mathbb{R}}^{j}$ for each integer $j$ satisfying $1 \leq j \leq m-n$ and $f$ is represented as the composition of $f_j$ with the canonical projection to $N$.
 \end{enumerate}
\end{Thm}

We explain about the family $\{f\} \bigcup \{f_j\}_{j=1}^{m-n}$ of maps in Theorem \ref{thm:7}. In Proposition \ref{prop:2} (\ref{prop:2.2}), we have a map $f$  
enjoying the following properties. A ({\it $k$-dimensional}) {\it hemisphere} means the set of all points of the $k$-dimensional unit sphere whose $k^{\prime}$-th component is for some integer $1 \leq k^{\prime} \leq k+1$. We can naturally have a canonical projection.
\begin{itemize}
	\item For some collar neighborhood $N(\partial W_f)$ of the boundary $\partial W_f \subset W_f$, the composition of $f {\mid}_{f^{-1}(N(\partial W_f))}:f^{-1}(N(\partial W_f)) \rightarrow N(\partial W_f)$ with the canonical projection to $\partial W_f$ gives the structure of a trivial linear bundle whose fiber is the unit disk $D^{m-n+1}$. 
	\item The restriction of $f$ to $f^{-1}(\partial W_f-{\rm Int}\ N(\partial W_f))$ gives the structure of a trivial smooth bundle over the complementary set $\partial W_f-{\rm Int}\ N(\partial W_f)$ of the interior ${\rm Int}\ N(\partial W_f) \subset W_f$ of $N(\partial W_f)$ whose fiber is the unit sphere $S^{m-n}$.
\end{itemize}

Furthermore, we glue these two bundles between the boundaries by isomorphism of the trivial bundles defined on the boundaries in a canonical way. 

\begin{itemize}
	\item We can choose the isomorphism of bundles as the product map of the diffeomorphisms to obtain a good special generic map $f_0$ on some suitable manifold $M_0$. 
	\item Let $W_{f_0}:=W_f$. 
	On $N(\partial W_f)$, denoted by $N(\partial W_{f_0})$, $f_0$ is regarded as the product map of a copy of a canonical projection of a hemisphere of the unit sphere $S^{m-n+1}$ into $\mathbb{R}$ and the identity map on the boundary $\partial W_{f_0}$ of $W_{f_0}$.
	\end{itemize}

	By the following procedure,	we can construct a new special map $f_j:M_0 \rightarrow N \times {\mathbb{R}}^{j}$ where $1 \leq j \leq m-n$ is an integer.

\begin{itemize}

	\item For the restriction of $f_0$ to ${f_0}^{-1}(\partial W_{f_0}-{\rm Int}\ N(\partial W_{f_0}))$, we replace $\partial W_{f_0}-{\rm Int}\ N(\partial W_{f_0})$ by $(\partial W_{f_0}-{\rm Int}\ N(\partial W_{f_0})) \times {\mathbb{R}}^j$ and replace the projection of the trivial bundle to the product map of a copy of a canonical projection of the unit sphere $S^{m-n}$ to ${\mathbb{R}}^j$ and the identity map on $\partial W_{f_0}-{\rm Int}\ N(\partial W_{f_0})$.
	\item For the restriction of $f_0$ to ${f_0}^{-1}(N(\partial W_{f_0}))$, we replace the original product map by the product map of a copy of a canonical projection of a hemisphere of the unit sphere $S^{m-n+1}$ into ${\mathbb{R}}^{j+1}$ and the identity map on $\partial W_{f_0}$.
\end{itemize}

This can be also constructed as a map such that the composition of $f_j$ with the canonical projection to $N$ is the original map $f:=f_0$.

This gives a desired family of special generic maps.
\begin{Def}
	\label{def:1}
we call $f=f_0$ enjoying the property here a {\it product-organized} special generic map.
\end{Def}

\section{Main Theorems.}
Let $A$ be a principal ideal domain having a unique identity element different from the zero element.
For a compact and smooth, PL, piecewise smooth, or topological manifold $X$, a homology class $a \in H_j(X,\partial X;A)$ is {\it represented} by a closed and connected submanifold $Y$ of dimension $j$ if for a suitable generator $\nu$ of $H_j(Y,\partial Y;A)$ and a suitable embedding $i:Y \rightarrow X$ satisfying $i({\rm Int}\ Y) \subset {\rm Int}\ X$ and $i(\partial Y) \subset \partial X$ chosen in the category where we discuss, $a=i_{\ast}(\nu)$. Such a generator is, a so-called {\it fundamental class}, canonically obtained for $Y$ if $Y$ is oriented or $A$ is a commutative group of order $2$.

In Poincar\'e duality theorem for a compact and connected (oriented) manifold $X$, for each homology class $c_{\rm h} \in H_j(X,\partial X;A)$ and $c_{\rm h} \in H_j(X;A)$, we have the {\it Poincar\'e
	duals} ${\rm PD}(c_{\rm h})$ to $c_{\rm h}$ as elements of $H^{\dim X-j}(X;A)$ and $H^{\dim X-j}(X,\partial X;A)$ uniquely. For a compact and connected (oriented) manifold $X$, for each cohomology class $c_{\rm c} \in H^j(X,\partial X;A)$ and $c_{\rm c} \in H^j(X;A)$, we have the {\it Poincar\'e
	duals} ${\rm PD}(c_{\rm c})$ to $c_{\rm c}$ as elements of $H_{\dim X-j}(X;A)$ and $H_{\dim X-j}(X,\partial X;A)$ uniquely.

We explain about another important notion. Let $A$ be a commutative ring having the unique identity element $1$ and the zero element $0$. For a free submodule of a homology group $H_j(X,X^{\prime};A)$ where $X$ and $X^{\prime} \subset X$ are topological spaces, its basis consisting of elements which are not divisible by any elements which are not units and an element $e \in H_j(X,X^{\prime};A)$ of the basis, we can define the element $e^{\ast} \in H^j(X,X^{\prime};A)$ enjoying the following two.

\begin{itemize}
	\item $e^{\ast}(e)=1$. 
	\item $e^{\ast}(e^{\prime})=0$ for any element $e^{\prime}$ of the basis different from $e$.
\end{itemize}
	 This is the {\it cohomology dual} to $e$. We omit expositions on the bases where we can guess easily.

\begin{proof}[Proof of Main Theorem \ref{mthm:1}]
	We prove (\ref{mthm:1.1}).
	
	Let $r \in \mathbb{Z}$.
	We define and investigate a compact, connected and smooth manifold $D_{r}$ of dimension $n>0$ smoothly embedded in ${\mathbb{R}}^n$ to have a special generic map $f_r:M_r \rightarrow {\mathbb{R}}^n$ in Proposition \ref{prop:2} (\ref{prop:2.2}) such that $W_{f_r}$, denoting "$W_f$ in Proposition \ref{prop:2}", is identified with $D_{r}$.
%【REVISE0】 "  investigate" → " investigate" .
%【REVISE】

$k_1,k_2>0$ are given as integers satisfying the relations $n=k_1+k_2$ and $2k_1<n$.
We have an $n$-dimensional smooth compact submanifold $D_{k_1,k_2}$ which is diffeomorphic to $S^{k_1} \times D^{k_2}$ and smoothly embedded in ${\mathbb{R}}^n$. 
$H_j(D_{k_1,k_2};\mathbb{Z})$ is the trivial group for $j \neq 0,k_1$ and isomorphic to $\mathbb{Z}$ for $j=0,k_1$. 
Let ${\nu}_{k_1,k_2}$ be a generator of $H_{k_1}(D_{k_1,k_2};\mathbb{Z}) \cong \mathbb{Z}$.

For $r \in \mathbb{Z}$, let $S_{r}$ be a $k$-dimensional closed manifold where $k$ is greater than or equal to $k_1$ and smaller than or equal to $\frac{n}{2}$. We also assume the existence of an element $a \in H_{k-k_1}(S_{r};\mathbb{Z})$ such that we can define the cohomology dual to it, considered as a generator of a free subgroup of rank $1$, and that $a$ is represented by a ($k-k_1$)-dimensional closed, connected and smooth submanifold $F_{r} \subset S_{r}$ with no boundary.

%We review the {\it dual} of a suitable element of a module $A$ over a commutative ring $R$ having a unique identity element different $1$ from the zero element.
%Consider an element $a \in A$ satisfying the following properties.

%\begin{enumerate}
%\item For any non-zero element $r$, $ra$ is not zero.
%\item We cannot represent $a=r^{\prime}a^{\prime}$ for any $r^{\prime} \in R$ which is not a unit and any element $a^{\prime} \in A$.
%\end{enumerate}
%Assume also that we can define a homomorphism $a^{\ast}:A \rightarrow R$ over $R$ satisfying the following two.
%\begin{enumerate}
%\item $a^{\ast}(a)=1$.
%\item For any arbitrary submodule $B$ of $A$ such that $A$ is the internal direct sum of the submodule generated by $\{a\} \subset A$ and $B$, $a^{\ast}(b)=0$ for any $b \in B$.
%\end{enumerate}
%In this situation, we can define such a homomorphism $a^{\ast}$ in a unique way and this is the {\it dual} of $a$. As a specific case, we can define the dual of such a  homology class as a cohomology class.
%【REVISE0】 We have deleted our exposition on duals.

To simplify the problem, let $S_{r}$ be diffeomorphic to $S^{k_1} \times S^{k-k_1}$ in the case $k-k_1>0$ and $S^{k_1}$ in the case $k=k_1$. $F_r$ is regarded as $\{\ast\} \times S^{k-k_1} \subset S^{k_1} \times S^{k-k_1}$ for some point $\ast \in S^{k_1}$ in the case $k-k_1>0$ and some point of $S^{k_1}=S^k$ in the case $k=k_1$.
 
An element represented by $F_{r}$ is a generator of a free subgroup of rank $1$ of $H_{k-k_1}(S_{r};\mathbb{Z})$ and we can have a basis of $H_{k-k_1}(S_{r};\mathbb{Z})$ containing this element. We have the cohomology dual to a homology class represented by $F_{r}$ in $S_{r}$.
For $r \in \mathbb{Z}$, we can take a smooth embedding $i_{k_1,k_2,r}:S_{r} \rightarrow {\rm Int}\ D_{k_1,k_2} \subset D_{k_1,k_2}$ such that a normal bundle is trivial and that $r{\nu}_{k_1,k_2}$ is the value of the homomorphism ${i_{k_1,k_2,r}}_{\ast}:H_{k_1}(S_{r};\mathbb{Z}) \rightarrow H_{k_1}(D_{k_1,k_2};\mathbb{Z})$ at the Poincar\'e dual to the element regarded as the cohomology dual to the element represented by $F_{r}$ in $S_{r}$ discussed just before. 
 We remove the interior of some small closed tubular neighborhood $N(S_{r})$ of the submanifold. 
 $N(S_{r})$ is regarded as the total space of a trivial linear bundle over $S_{r}$ whose fiber is the unit disk $D^{n-k}$.
 Let $D_{r}$ denote the resulting $n$-dimensional submanifold in
${\mathbb{R}}^n$. We also have a Mayer-Vietoris sequence \\
$$\rightarrow H_j(\partial D_{r} \bigcap N(S_{r});\mathbb{Z}) \rightarrow H_j(D_{r};\mathbb{Z}) \oplus H_j(N(S_{r});\mathbb{Z}) \rightarrow H_j(D_{k_1,k_2};\mathbb{Z}) \rightarrow$$
and $\partial D_{r} \bigcap N(S_{r})$ is regarded as the total space of a trivial linear bundle over $S_{r}$ whose fiber is the unit sphere $S^{n-k-1}$ and which is the subbundle of the bundle $N(S_{r})$.
We investigate the case where $0 \leq j<k_1$. The homomorphism from $H_{j}(D_{k_1,k_2};\mathbb{Z}) \cong \{0\}$ into $H_{j-1}(\partial D_{r} \bigcap N(S_{r});\mathbb{Z})$ for $1 \leq j <k_1$ is a monomorphism.

 $S_{r}$ is diffeomorphic to $S^{k_1} \times S^{k-k_1}$ and the fiber of the trivial bundle $N(S_{r})$ over $S_{r}$ is the unit disk $D^{n-k}$ and the fiber of the trivial bundle $\partial D_r \bigcap N(S_{r})$ over $S_{r}$ is the unit sphere $S^{n-k-1}$.
 We can see that for $0 \leq j \leq k_1-1<\frac{n}{2}-1 \leq n-k-1$, each element of $H_{j}(\partial D_{r} \bigcap N(S_{r});\mathbb{Z})$ is represented by a submanifold in the image of a section of the trivial bundle $\partial D_{r} \bigcap N(S_{r})$ over $S_{r}$. We can also see that 
the homomorphism from $H_{j}(\partial D_{r} \bigcap N(S_{r});\mathbb{Z})$ into $H_j(D_{r};\mathbb{Z}) \oplus H_j(N(S_{r});\mathbb{Z})$ is represented as the direct sum of the homomorphism onto $\{0\} \subset H_j(D_{r};\mathbb{Z})$ and an isomorphism onto $H_j(N(S_{r});\mathbb{Z})$ for $0 \leq j \leq k_1-1<\frac{n}{2}-1 \leq n-k-1$. We have $H_{j}(D_{r};\mathbb{Z}) \cong \mathbb{Z}$ for $j=0$ and $H_{j}(D_{r};\mathbb{Z}) \cong \{0\}$ for $0<j<k_1$.

%$D_{r}$ is orientable and the boundary $\partial D_{r}$ consists of exactly two connected components.

%We have a homology exact sequence for the pair $(D_r,\partial D_r)$, given by

%$$\rightarrow H_1(D_{r};\mathbb{Z}) \rightarrow H_1(D_r,\partial D_{r};\mathbb{Z}) \rightarrow H_0(\partial D_{r};\mathbb{Z}) \rightarrow H_0(D_{r};\mathbb{Z}) \rightarrow$$

%and here, we can see that $H_1(D_r,\partial D_r;\mathbb{Z})$ has a basis enjoying the following two.
%\begin{itemize}
%	\item Each of which is represented by a $1$-dimensional compact and connected manifold diffeomorphic to the $1$-dimensional unit disk $D^1$ or a closed interval.
%	\item The interiors of the $1$-dimensional manifolds before are smoothly embedded into the interior ${\rm Int}\ D_r$ and the boundaries of the $1$-dimensional manifolds are smoothly embedded into the boundary $\partial D_r$.
%	\item The size of the basis is exactly $1$. Note that the size is $l \geq 0$ if the boundary $\partial D_{r}$ consists of exactly $l+1$ connected components instead and that the previous two properties are enjoyed.   
%	\end{itemize} 

%We have $H^1(D_{r},\partial D_{r};\mathbb{Z}) \cong H^{n-1}(D_{r};\mathbb{Z}) \cong \mathbb{Z}$ by virtue of universal coefficient theorem and Poincar\'e duality theorem for $D_{r}$.
% We have exactly two generators of the group $H_{n-1}(D_{r};\mathbb{Z})$ and they are always represented by connected components of the boundary $\partial D_{r}$.
By the condition on the integers,
$1 \leq k_1<\frac{n}{2}<n-1$. This means that $H_{n-1}(D_{k_1,k_2};\mathbb{Z})$ is the trivial group. $k \leq \frac{n}{2}<n-1$ and $S_{r}$ is $k$-dimensional. $N(S_{r})$ collapses to $S_{r}$. We have $H_{n-1}(N(S_{r});\mathbb{Z}) \cong \{0\}$. We have $H_{n-1}(\partial D_{r} \bigcap N(S_{r});\mathbb{Z}) \cong \mathbb{Z}$.
By the sequence, we have $H_{n-1}(D_{r};\mathbb{Z}) \cong \mathbb{Z}$ and its generator is represented by the connected component of the boundary of some (small) collar neighborhood of $\partial D_{r} \bigcap N(S_{r})$ in the interior of $D_{r}$.

$N(S_{r})$ is the total space of a trivial linear bundle over $S_r$, diffeomorphic to $S^{k_1} \times S^{k-k_1}$, where its fiber is the unit disk $D^{n-k}$. The trivial bundle $D_{r} \bigcap N(S_{r})$ over $S_{r}$ over $S_{r}$ is the subbundle of the bundle $N(S_{r})$ whose fiber is the unit sphere $S^{n-k-1}=\partial D^{n-k}$.
We can see that for $j \neq 0,k_1,k-k_1,n-k-1,n-k+k_1-1,k,n-k_1-1,n-1$, $H_{j}(D_{r};\mathbb{Z})$ is the trivial group, since the first and the third groups in the sequence are the trivial groups.

This completes the proof of (\ref{mthm:1.1.1}) and (\ref{mthm:1.1.2}).

The homomorphism from $H_{k_1}(D_{k_1,k_2};\mathbb{Z}) \cong \mathbb{Z}$ into $H_{k_1-1}(\partial D_{r} \bigcap N(S_{r});\mathbb{Z})$ is a homomorphism onto $\{0\} \subset H_{k_1-1}(\partial D_{r} \bigcap N(S_{r});\mathbb{Z})$ since the homomorphism from $H_{k_1-1}(\partial D_{r} \bigcap N(S_{r});\mathbb{Z})$ into $H_{k_1-1}(N(S_{r});\mathbb{Z})$ in the sequence is also an isomorphism onto $H_{k_1-1}(N(S_{r});\mathbb{Z})$ by an argument before. The homomorphism from $H_{j}(D_{k_1,k_2};\mathbb{Z})$ into $H_{j-1}(\partial D_{r} \bigcap N(S_{r});\mathbb{Z})$ is a monomorphism onto $\{0\} \subset H_{j-1}(\partial D_{r} \bigcap N(S_{r});\mathbb{Z})$ for $j>k_1$. 
$H_{k_1}(N(S_{r});\mathbb{Z}) \cong \mathbb{Z}$ for $k_1 \neq k-k_1$ and $H_{k_1}(N(S_{r});\mathbb{Z}) \cong \mathbb{Z} \oplus \mathbb{Z}$ for $k_1 = k-k_1$.
We can also see that the homomorphism from $H_{j}(\partial D_{r} \bigcap N(S_{r});\mathbb{Z})$ into $H_{j}(D_{r};\mathbb{Z}) \oplus H_{j}(N(S_{r});\mathbb{Z})$ is a monomorphism and represented as the direct sum of the homomorphism into $H_{j}(D_{r};\mathbb{Z})$ and the homomorphism into $H_{j}(N(S_{r});\mathbb{Z})$ which is also an epimorphism for $j \geq 0$.

By reviewing our arguments on the sequence, the sum of the ranks of $H_{j}(D_{r};\mathbb{Z})$ and $H_{j}(N(S_{r});\mathbb{Z})$ and that of the ranks of $H_{j}(\partial D_{r} \bigcap N(S_{r});\mathbb{Z})$ and $H_j(D_{k_1,k_2};\mathbb{Z})$ agree. $H_{j}(N(S_{r});\mathbb{Z})$, $H_{j}(\partial D_{r} \bigcap N(S_{r});\mathbb{Z})$ and $H_j(D_{k_1,k_2};\mathbb{Z})$ are free. 
The cokernel of the homomorphism from $H_{j}(\partial D_{r} \bigcap N(S_{r});\mathbb{Z})$ into the direct sum $H_{j}(D_{r};\mathbb{Z}) \oplus H_{j}(N(S_{r});\mathbb{Z})$ is isomorphic to $H_j(D_{k_1,k_2};\mathbb{Z})$. Thus $H_{j}(D_{r};\mathbb{Z})$ is shown to be free. In the case $j=k_1 \leq n-k-1$, the rank of $H_j(D_{k_1,k_2};\mathbb{Z})=H_{k_1}(D_{k_1,k_2};\mathbb{Z})$ is $1$. Thus $H_{k_1}(D_{r};\mathbb{Z})$ is free and not the trivial group. 
Remember that $D_{r} \bigcap N(S_{r})$ is diffeomorphic to $S^{k_1} \times S^{k-k_1} \times S^{n-k-1}=S^{k_1} \times S^{k-k_1} \times \partial D^{n-k}$ and $N(S_r)$ is diffeomorphic to $S^{k_1} \times S^{k-k_1} \times D^{n-k} \supset S^{k_1} \times S^{k-k_1} \times \partial D^{n-k}$.
We have the equations $n-k_1-1=(n-k-1)+(k-k_1)$ and $n-k+k_1-1=(n-k-1)+k_1$.
In the case $j=n-k-1,n-k_1-1,n-k+k_1-1$, the rank of $H_j(N(S_{r});\mathbb{Z})$ is smaller than that of $H_{j}(\partial D_{r} \bigcap N(S_{r});\mathbb{Z})$ and we also have that $H_{j}(D_{r};\mathbb{Z})$ is free and not the trivial group. 
From the relation $1 \leq k_1<\frac{n}{2}$, we have $k_1 \leq k \leq \frac{n}{2}<n-1$. In the case $j=k-k_1,k$, satisfying $j \notin \{k_1, n-k-1,n-k_1-1,n-k+k_1-1\}$, the homomorphism from $H_{j}(\partial D_{r} \bigcap N(S_{r});\mathbb{Z})$ into the direct sum $H_{j}(D_{r};\mathbb{Z}) \oplus H_{j}(N(S_{r});\mathbb{Z})$ and the homomorphism into $H_{j}(N(S_{r});\mathbb{Z})$ here are isomorphisms and $H_{j}(D_{r};\mathbb{Z})$ is the trivial group. In this case, $j \neq n-1$ also holds: $H_{n-1}(D_{r};\mathbb{Z})$ is shown to be isomorphic to $\mathbb{Z}$.

This completes the proof of (\ref{mthm:1.1.3}).
% Assume that $k=k_1$ here. Suppose also $n-k-1=k_1=k<\frac{n}{2}$. $H_{k_1}(\partial D_{r} \bigcap N(S_{r});\mathbb{Z})$ is isomorphic to $\mathbb{Z} \oplus \mathbb{Z}$. By the sequence and that the homomorphism from $H_{k_1}(D_{k_1,k_2};\mathbb{Z}) \cong \mathbb{Z}$ into $H_{k_1-1}(\partial D_{r} \bigcap N(S_{r});\mathbb{Z})$ is a homomorphism onto $\{0\} \subset H_{k_1-1}(\partial D_{r} \bigcap N(S_{r});\mathbb{Z})$, $H_{k_1}(D_{r};\mathbb{Z})$ is isomorphic to $\mathbb{Z} \oplus \mathbb{Z}$. Suppose that $n-k-1 \neq k_1=k<\frac{n}{2}$ instead here. This assumption means $n-k-1>k_1=k$. By the sequence and that the homomorphism from $H_{k_1}(D_{k_1,k_2};\mathbb{Z}) \cong \mathbb{Z}$ into $H_{k_1-1}(\partial D_{r} \bigcap N(S_{r});\mathbb{Z})$ is a homomorphism onto $\{0\} \subset H_{k_1-1}(\partial D_{r} \bigcap N(S_{r});\mathbb{Z})$, $H_{k_1}(D_{r};\mathbb{Z})$ is isomorphic to $\mathbb{Z}$. By the sequence and that the homomorphism from $H_{n-k-1}(\partial D_{r} \bigcap N(S_{r});\mathbb{Z}) \cong \mathbb{Z}$ in the sequence is a monomorphism and represented as the direct sum of the homomorphism onto $H_{n-k-1}(\partial D_{r};\mathbb{Z})$ and the zero homomorphism onto $H_{n-k-1}(N(S_{r});\mathbb{Z}) \cong \{0\}$. $H_{n-k-1}(D_{k_1,k_2};\mathbb{Z})$ is the trivial group. By the sequence, $H_{n-k-1}(D_{r};\mathbb{Z})$ is isomorphic to $\mathbb{Z}$. This completes the proof of (\ref{mthm:1.1.3}).

We prove (\ref{mthm:1.1.4}). In the case $k_1>1$, $D_{r}$ is simply-connected, by the fact that the fundamental groups ($1$st homotopy groups) ${\pi}_1(\partial D_{r} \bigcap N(S_{r}))$ and ${\pi}_1(N(S_{r}))$ are commutative (since the spaces are diffeomorphic or simple homotopy equivalent to products of finitely many spheres), by the fact that $D_{k_1,k_2}$ is simply-connected and obtained by gluing $N(S_{r})$ and $D_{r}$ on the boundaries by a diffeomorphism and by virtue of Seifert-van Kampen theorem. This completes the proof of (\ref{mthm:1.1.4}). 

This completes the proof of (\ref{mthm:1.1}).

We show (\ref{mthm:1.2}).

$H_{k_1}(D_{r};\mathbb{Z})$ is free and not the trivial group.

%Suppose that $n-k-1 \neq k_1<\frac{n}{2}$ instead here. %This assumption means $n-k-1 \geq n-\frac{n}{2}-1 \geq k_1$. 

%By the sequence the homomorphism from $H_{k_1}(D_{k_1,k_2};\mathbb{Z}) \cong \mathbb{Z}$ into $H_{k_1-1}(\partial D_{r} \bigcap N(S_{r});\mathbb{Z})$ is a homomorphism onto $\{0\} \subset H_{k_1-1}(\partial D_{r} \bigcap N(S_{r});\mathbb{Z})$, $H_{k_1}(D_{r};\mathbb{Z})$ is isomorphic to $\mathbb{Z} \oplus \mathbb{Z}$ in the case $k=2k_1$ and $\mathbb{Z}$ in the case $k \neq 2k_1$. By the sequence and that the homomorphism from $H_{n-k-1}(\partial D_{r} \bigcap N(S_{r});\mathbb{Z}) \cong \mathbb{Z}$ in the sequence is a monomorphism and represented as the direct sum of the homomorphism onto $H_{n-k-1}(\partial D_{r};\mathbb{Z})$ and the zero homomorphism onto $H_{n-k-1}(N(S_{r});\mathbb{Z}) \cong \{0\}$. $H_{n-k-1}(D_{k_1,k_2};\mathbb{Z})$ is the trivial group. By the sequence, $H_{n-k-1}(D_{r};\mathbb{Z})$ is isomorphic to $\mathbb{Z}$. This completes the proof of (\ref{mthm:1.1.3}).

We remark on the cohomology rings of the $m$-dimensional manifolds of the domains of the special generic maps in $\{f_r:M_r \rightarrow {\mathbb{R}}^n\}$. 

Remember that we have special generic maps as in Proposition \ref{prop:2} (\ref{prop:2.2}). This
 makes trivial smooth bundles and enables us to take sections.  

Remember also that according to Proposition \ref{prop:3} (\ref{prop:3.5}), $f_r$ induces isomorphisms between the two homology groups $H_{j}(M_{r};\mathbb{Z})$ and $H_{j}(D_{r}:\mathbb{Z})$ and the two cohomology groups $H^{j}(M_{r};\mathbb{Z})$ and $H^{j}(D_{r}:\mathbb{Z})$ for $0 \leq j \leq m-n$. 

We can take another small closed tubular neighborhood $N_0(S_{r})$ of $i_{k_1,k_2,r}(S_{r})$
satisfying $i_{k_1,k_2,r}(S_{r}) \subset N(S_{r}) \subset {\rm Int}\ N_0(S_{r}) \subset N_0(S_{r})$. The original closed tubular neighborhood is in the interior of this new closed tubular neighborhood and the boundary of the new closed tubular neighborhood, which is connected since the relations $k \leq \frac{n}{2}$ and $n \geq 3$ hold, is regarded as the boundary of some small collar neighborhood of $\partial D_{r} \bigcap N(S_{r})$. Moreover, a generator of $H_{n-1}(D_{r};\mathbb{Z}) \cong \mathbb{Z}$ is, as presented in the proof of (\ref{mthm:1.1.1}) and (\ref{mthm:1.1.2}), represented by this boundary.

 In a suitable way, we may regard ${f_r}^{-1}(\partial N_0(S_{r}))$ as the total space of a trivial smooth bundle over $S^{k_1} \times \{{\ast}^{\prime}\} \subset S^{k_1} \times S^{k-k_1}$ whose fiber is diffeomorphic to $S^{k-k_1} \times \partial D^{n-k} \times S^{m-n}$ in the case $k>k_1$ and $\partial D^{n-k} \times S^{m-n}$ in the case $k=k_1$ for some point ${\ast}^{\prime} \in S^{k-k_1}$. 

In the case $k_1=n-k_1-1$, we can define the basis of $H_{k_1}(M_{r};\mathbb{Z})=H_{n-k_1-1}(M_{r};\mathbb{Z})$ for defining cohomology duals consists of exactly two elements in the following and some additional elements if we need.
\begin{itemize}
	\item An element which is represented by the image of the section of the trivial smooth bundle ${f_r}^{-1}(\partial N_0(S_{r}))$ over $S^{k_1} \times \{{\ast}^{\prime}\}$.
	\item An element which is represented by a submanifold $S^{k-k_1} \times \partial D^{n-k} \times \{{\ast}^{\prime \prime}\} \subset S^{k-k_1} \times \partial D^{n-k} \times S^{m-n}$ in the case $k>k_1$ or $\partial D^{n-k} \times \{{\ast}^{\prime \prime}\} \subset \partial D^{n-k} \times S^{m-n}$ in the case $k=k_1$ for some point ${\ast}^{\prime \prime}$.
\end{itemize}	

We can argue similarly in the case $k_1 \neq n-k_1-1$. The ranks of the groups we consider are $1$.

We investigate the cup product of the cohomology dual to the homology class of degree $k_1$ and the cohomology dual to the homology class of degree $n-k_1-1$. We can easily see that this cup product is $r$ times the cohomology dual to a suitable homology class of degree $n-1$ to which we can define the cohomology dual.

%【REVISE】 For duals, we use prepositions "to". We have added expositions on bases of (subgroups of) homology groups

More precisely, this homology class of degree $n-1$ is represented by the image of a section of
the trivial smooth bundle ${f_r}^{-1}(\partial N_0(S_{r}))$ over $\partial N_0(S_{r})$ given by $f_r$. The fiber of the bundle is the unit sphere $S^{m-n}$.
Remember also that a generator of $H_{n-1}(D_{r};\mathbb{Z}) \cong \mathbb{Z}$ is regarded as an element represented by $\partial N_0(S_{r})$.
This means that the homology class of degree $n-1$ represented by the image of the section of the trivial smooth bundle before is an element which is not divisible by any integer greater than $1$.

The relation $k_1,n-k_1-1 \leq m-n+1$ is assumed. We have $k_1+(n-k_1-1)=n-1$. This with Theorem \ref{thm:5} implies that $M_r$ admits no special generic maps into ${\mathbb{R}}^{n^{\prime}}$ for $1 \leq n^{\prime}<n$ for $r \neq 0$. This completes the proof of (\ref{mthm:1.2}).

Furthermore, by taking product-organized special generic maps in Definition \ref{def:1}, we also have (\ref{mthm:1.3}).

We show (\ref{mthm:1.4}). $N(S_{r})$ is the trivial linear bundle over $S_{r}$, diffeomorphic to $S^{k_1} \times S^{k-k_1}$ ($k>k_1$) and $S^{k_1}$ ($k=k_1$), whose fiber is the unit disk $D^{n-k}$. $D_{r} \bigcap N(S_{r})$ is the subbundle whose fiber is $\partial D^{n-k}=S^{n-k-1}$. In the case $k_1=n-k-1$, the rank of $H_{k_1}(D_{r} \bigcap N(S_{r});\mathbb{Z})$ is greater than that of $H_{k_1}(N(S_{r});\mathbb{Z})$ by $1$, the rank of $H_{k_1}(D_{k_1,k_2};\mathbb{Z})$ is $1$ and
by the argument on the ranks before, the rank of $H_{k_1}(D_{r};\mathbb{Z})=H_{n-k-1}(D_{r};\mathbb{Z})$ is shown to be $2$. In the case $k_1 \neq n-k-1$, $k_1<n-k-1$ holds,
the ranks of $H_{k_1}(D_{r} \bigcap N(S_{r});\mathbb{Z})$ and $H_{k_1}(N(S_{r});\mathbb{Z})$ agree and the rank of $H_{k_1}(D_{r};\mathbb{Z})$ is shown to be $1$. In addition to this argument on ranks of homology groups, remember the argument on cup products and an element of degree $n-1$ of $H^{n-1}(M_r;\mathbb{Z})$ which is not divisible by any integer greater than $1$ before.
 Proposition \ref{prop:3} (\ref{prop:3.5}) completes the proof of (\ref{mthm:1.4}).

This completes the proof.
\end{proof}
\begin{Ex}
\label{ex:2}

	Let $m \geq 6$ be an arbitrary integer.
$S^2 \times S^{m-4}$ admits a special generic map into ${\mathbb{R}}^3$ in Example \ref{ex:1} and restrict the manifold of the target to a copy of the $3$-dimensional unit disk smoothly embedded in the original manifold of the target and containing the original image in the interior. 
%【REVISE】 ex:3 → ex:2 .
%【REVISE】 restrict the target → restruct the manifold of the manifold .
%【REVISE】 embedded in the target → embedded in the original manifold of the target .
We consider the product map of this and the identity map on $S^2$. We can embed the manifold of the target into ${\mathbb{R}}^5$ smoothly to have a special generic map. This is a case for $(n,r,k_1,k)=(5,1,2,2)$ in Main Theorem \ref{mthm:1}.

$k_1=n-k-1=2$ here and the relation $k_1,n-k_1-1=2 \leq m-n+1=m-4$ is satisfied. For example, by changing $r \neq 0$ and taking sufficiently high $m$, we have a family of $m$-dimensional closed and simply-connected manifolds whose rational cohomology rings are isomorphic to $H^{\ast}(S^2 \times S^{2} \times S^{m-4};\mathbb{Q})$ and which admit no special generic maps into ${\mathbb{R}}^{n^{\prime}}$ for any $n^{\prime}=1,2,3,4$. 
%【REVISE0】 explains → is
%【REVISE0】 this target → the manifold of the target

\end{Ex}
\begin{Rem}
With a little effort, we will give a proof of the fact that maps in Main Theorem \ref{mthm:1} are obtained via technique of lifting smooth maps in \cite{kitazawa5, kitazawa9,kitazawa10,kitazawa11,kitazawa12,kitazawa15}. For such lifting, see also \cite{kitazawa6,kitazawa7,kitazawa8}. 
%【REVISE0】 \cite{kitazawa9}, \cite{kitazawa10}, \cite{kitazawa11} and \cite{itazawa12} → \cite{kitazawa9, kitazawa10,kitazawa11,kitazawa12,kitazawa15} .
 More precisely, we consider fold
maps which may not be special generic and preimages of whose regular values are disjoint unions of standard spheres and construct special generic maps so that the compositions with suitable canonical projections are the given maps in these papers.
%【REVISE】 We changed this sentence as the following: "More precisely, we consider fold maps which may not be special generic and preimages of whose regular values are disjoint unions of standard spheres and construct special generic maps so that the compositions with suitable canonical projections are the given maps in these papers."
 This proof will be given in some forthcoming articles. 
\end{Rem}
%【REVISE】 We have deleted this remark, adding Theorem \ref{thm:7} instead. 
\begin{Prop}
	\label{prop:4}
	Let $B$ be the total space of a linear bundle over the real projective plane satisfying the following two conditions.
	\begin{enumerate}
		\item The fiber is a $2$-dimensional standard sphere.
		\item $B$ is a $4$-dimensional, closed and orientable manifold.
		\end{enumerate}
	$H_j(B;\mathbb{Z})$ is isomorphic to $\mathbb{Z}/2\mathbb{Z}$ for $j=1,2$, the trivial group for $j=3$ and $\mathbb{Z}$ for $j=0,4$.
\end{Prop}
We have such a manifold as the boundary of a closed tubular neighborhood of a copy of the real projective plane smoothly embedded in ${\mathbb{R}}^5$. We consider this $4$-dimensional manifold later. 
\begin{proof}[Proof of Proposition \ref{prop:4}]
% The bundle $B$ admits a section by virtue of a kind of theorems from classical fundamental theory of obstruction. Related theory is presented in \cite{milnorstasheff} and a short exposition is given before Proposition \ref{prop:7}, later. 
$B$ is obtained by gluing the total space $B_1$ of a linear bundle over the M\"obius band whose fiber is $S^2$ and the total space $B_2:=D^2 \times S^2$ of a trivial linear bundle over $D^2$ via a bundle isomorphism between the linear bundles obtained by the restrictions to the boundaries. We have the following Mayer-Vietoris sequence
$$\rightarrow H_j(B_1 \bigcap B_2;\mathbb{Z}) \rightarrow H_j(B_1;\mathbb{Z}) \oplus H_j(B_2;\mathbb{Z}) \rightarrow H_j(B;\mathbb{Z}) \rightarrow$$
and we can see that $H_j(B_1;\mathbb{Z})$, $H_j(B_2;\mathbb{Z})$, and $H_j(B_1 \bigcap B_2;\mathbb{Z})$ are isomorphic to $\mathbb{Z}$ for $j=1,2$ except for $H_1(B_2;\mathbb{Z})$, isomorphic to the trivial group.

For $j=1,2$ the homomorphism from $H_j(B_1 \bigcap B_2;\mathbb{Z})$ into $H_j(B_1;\mathbb{Z}) \oplus H_j(B_2;\mathbb{Z})$ is represented as the direct sum of two homomorphisms. More precisely, we have the following two.
\begin{itemize}
\item For $j=1$, the homomorphism is represented as the direct sum of a monomorphism mapping a generator ${\nu}_{0,1} \in H_1(B_1 \bigcap B_2;\mathbb{Z})$ to an element represented as $2{\nu}_1={\nu}_1+{\nu}_1$ for a generator ${\nu}_1 \in H_1(B_1;\mathbb{Z})$ and the zero homomorphism into $H_1(B_2;\mathbb{Z})$.
\item For $j=2$, the homomorphism is represented as the direct sum of a monomorphism mapping a generator ${\nu}_{0,2} \in H_2(B_1 \bigcap B_2;\mathbb{Z})$ to an element represented as $2{\nu}_2={\nu}_2+{\nu}_2$ for a generator ${\nu}_2 \in H_2(B_1;\mathbb{Z})$ and an isomorphism onto $H_2(B_2;\mathbb{Z})$. We have $2{\nu_2} \in H_2(B_1;\mathbb{Z})$ here due to the fact that $B_1$ is orientable as a $4$-dimensional compact manifold whereas the base space is diffeomorphic to the M\"obius band and not orientable.
\end{itemize}
We have that $H_j(B;\mathbb{Z})$ is isomorphic to $\mathbb{Z}/2\mathbb{Z}$ for $j=1,2$, the trivial group for $j=3$ and $\mathbb{Z}$ for $j=0,4$.
\end{proof}

Without investigating the homomorphism for $j=2$, we can also have the result of Proposition \ref{prop:4}. We can see that the Euler number of $B$ is that of the product of the real projective plane and that of $S^2$ and $2$. $B$ is orientable. Poincar\'e duality theorem completes the proof.

\begin{Prop}
	\label{prop:5}
	Let ${D_0}^5$ be a manifold diffeomorphic to the $5$-dimensional unit disk $D^5$, $S^{\prime}$ a copy of the real projective plane smoothly embedded in the interior ${\rm Int}\ {D_0}^5$ and $D^{\prime}$ the manifold obtained by removing the interior
	of a small closed tubular neighborhood $N(S^{\prime})$ of $S^{\prime}$. Then
$H_j(D^{\prime};\mathbb{Z})$ is isomorphic to $\mathbb{Z}$ for $j=0,4$, $\mathbb{Z}/2\mathbb{Z}$ for $j=2$ and the trivial group for $j \neq 0,2,4$. Furthermore, $D^{\prime}$ is simply-connected.
\end{Prop}

%【REVISE】 $S^{\prime}$ be a copy → $S^{\prime}$ a copy .
\begin{proof}[Proof of Proposition \ref{prop:5}]
  We have a Mayer-Vietoris sequence
 $$\rightarrow H_j(D^{\prime} \bigcap N(S^{\prime});\mathbb{Z}) \rightarrow H_j(D^{\prime};\mathbb{Z}) \oplus H_j(N(S^{\prime});\mathbb{Z}) \rightarrow H_j({D_0}^5;\mathbb{Z}) \rightarrow$$
 and we have the following two.
 \begin{itemize}
\item $D^{\prime} \bigcap N(S^{\prime})$ is diffeomorphic to 
the total space of a linear bundle over the real projective plane satisfying the following two conditions.
\begin{itemize}
	\item The fiber is a $2$-dimensional standard sphere.
	\item The total space is a $4$-dimensional, closed and orientable manifold.
\end{itemize}
 In a word, it is regarded as a manifold $B$ in Proposition \ref{prop:4}. Furtherore, we can easily see that ${\pi}_1(B) \cong \mathbb{Z}/2\mathbb{Z}$.
\item $H_j(N(S^{\prime});\mathbb{Z})$ is isomorphic to $\mathbb{Z}$ for $j=0$, $\mathbb{Z}/2\mathbb{Z}$ for $j=1$ and the trivial group for $j \neq 0,1$. We can easily see that ${\pi}_1(N(S^{\prime})) \cong \mathbb{Z}/2\mathbb{Z}$. The homomorphism from $H_j(D^{\prime} \bigcap N(S^{\prime});\mathbb{Z})$ into $H_j(N(S^{\prime});\mathbb{Z})$ is an isomorphism for $j=0,1$ and the zero homomorphism for $j \neq 0,1$.
\end{itemize}
For example by these arguments we can know the following properties of the homomorphisms.
\begin{itemize}
\item The direct sum of the homomorphism from $H_j(D^{\prime} \bigcap N(S^{\prime});\mathbb{Z})$ into $H_j(D^{\prime};\mathbb{Z})$ and the homomorphism from $H_j(D^{\prime} \bigcap N(S^{\prime});\mathbb{Z})$ into $H_j(N(S^{\prime});\mathbb{Z})$ in the Mayer-Vietoris sequence is isomorphism for $j \neq 0$.

\item The homomorphism from $H_j(D^{\prime} \bigcap N(S^{\prime});\mathbb{Z})$ into $H_j(D^{\prime};\mathbb{Z})$ in the sequence is the zero homomorphism for $j=1$.

\end{itemize}
We can see that $H_j(D^{\prime};\mathbb{Z})$ is isomorphic to $\mathbb{Z}$ for $j=0,4$, $\mathbb{Z}/2\mathbb{Z}$ for $j=2$ and the trivial group for $j \neq 0,2,4$. We can also see that $D^{\prime}$ is simply-connected by virtue of Seifert-van Kampen theorem and a similar argument in the proof of Main Theorem \ref{mthm:1}.
\end{proof}
The following proposition can be shown by fundamental calculations of (co)homology groups and is also shown in \cite{nishioka}. In the smooth, PL (piecewise smooth) and the topology category, the proofs are same.
%【REVISE】 We made another sentence.
%【REVISEonComment】 We made Propositions 4 and 5 for homology groups of some explicit important manifolds. We have drastically changed some expositions.
\begin{Prop}
\label{prop:6}
Let $A$ be a principal ideal domain having a unique identity element different from the zero element. Let $P$ be a compact and connected manifold satisfying $\dim P>3$ and $H_1(P;A)$ is the trivial group. Then, $P$ is orientable and $H_{j}(P;A)$ is free for $j=\dim P-2,\dim P-1$.
%【REVISE0】 We revised expositions on the category and the prientablity in this proposition.
%【REVISE】 $H_1(P;A)=\{0\}$ → $H_1(P;A)=\{0\}$ must be the trivial group
\end{Prop}

Regarding $P$ as "$W_f$ in Proposition \ref{prop:2}" with the classifications of $5$-dimensional closed and simply-connected manifolds of \cite{barden}, which is also in \cite{crowley}, is a key ingredient of \cite{nishioka} or the proof of Theorem \ref{thm:3}.

%\begin{proof}[Proof of Main Theorem \ref{mthm:2}]
%	The statement on the homology groups follows from the last property of the five properties of Proposition \ref{prop:3}. This property together with Proposition \ref{prop:4} proves the latter part and completes the proof.
%\end{proof}
We present Main Theorem \ref{mthm:2} in a more explicit form as the following.
\begin{Thm}
\label{thm:8}
Let $m>6$ be an integer. Let $f:M \rightarrow {\mathbb{R}}^5$ be a special generic map on an $m$-dimensional closed and simply-connected manifold such that "$W_f$ in Proposition {\rm \ref{prop:2}}" is diffeomorphic to $D^{\prime}$ in Proposition \ref{prop:5} smoothly embedded in ${\mathbb{R}}^5$ and obtained via Proposition {\rm \ref{prop:2}} {\rm (}{\rm \ref{prop:2.2}}{\rm )} using the embedding ${\bar{f}}_N$. In this situation, $H_j(M;\mathbb{Z}) \cong H_j(W_f;\mathbb{Z})$ for $0 \leq j \leq m-5$ and $M$ admits no special generic maps into ${\mathbb{R}}^{n}$ for $1 \leq n \leq 4$.

Furthermore, instead of the manifold $M$ and the map $f$, we can take another $m$-dimensional closed and simply-connected manifold $M_0$ and a special generic map $f_0:M_0 \rightarrow {\mathbb{R}}^n$ as a product-organized one in Definition \ref{def:1} such that $W_{f_0}=W_f$ where $W_{f_0}$ denotes "$W_f$ in Proposition \ref{prop:2}" for the new map $f_0$ and $M_0$ admits a special generic map into ${\mathbb{R}}^n$ if and only if $5 \leq n \leq m$.
\end{Thm}
\begin{proof}
We have a special generic map on an $m$-dimensional closed and simply-connected manifold $f:M \rightarrow {\mathbb{R}}^5$ such that $W_f$ in Proposition \ref{prop:2} is diffeomorphic to $D^{\prime}$ in Proposition \ref{prop:5}. By Proposition \ref{prop:3} (\ref{prop:3.5}), we have the isomorphism $H_j(M;\mathbb{Z}) \cong H_j(W_f;\mathbb{Z})$ for $0 \leq j \leq m-5$ and $H_2(M;\mathbb{Z})$ is not free. If $M$ admits a special generic map $f^{\prime}:M \rightarrow {\mathbb{R}}^4$, then the $4$-dimensional manifold $W_{f^{\prime}}$, defined as in Proposition \ref{prop:2}, must be one like $P$ in Proposition \ref{prop:4}. More precisely, $H_1(W_{f^{\prime}};\mathbb{Z})$ is the trivial group and $H_2(W_{f^{\prime}};\mathbb{Z})$ is not free by Proposition \ref{prop:3} (\ref{prop:3.5}) and this contradicts Proposition \ref{prop:6}, saying that $H_2(W_{f^{\prime}};\mathbb{Z})$ is free. 
If $M$ admits a special generic map $f^{\prime}:M \rightarrow {\mathbb{R}}^n$ for $n=1,2,3$, then the $n$-dimensional manifold $W_{f^{\prime}}$, defined as in Proposition \ref{prop:2}, must satisfy the fact that the groups $H_2(W_{f^{\prime}};\mathbb{Z}) \cong H_2(M;\mathbb{Z})$ are free. For this, 
note that Proposition \ref{prop:3} (\ref{prop:3.5}) is applied and that $P$ is simply-connected by Proposition \ref{prop:3} (\ref{prop:3.5}).

We can also construct a special generic map on an $m$-dimensional closed and simply-connected manifold $M_0$ as a product-organized one $f_0:M_0 \rightarrow {\mathbb{R}}^n$ in Definition \ref{def:1} instead of the manifold $M$ and the map $f:M \rightarrow {\mathbb{R}}^n$ such that $W_{f_0}=W_f$ where $W_{f_0}$ denotes "$W_f$ in Proposition \ref{prop:2}" for the new map $f_0$. $M_0$ admits a special generic map into ${\mathbb{R}}^n$ if and only if $5 \leq n \leq m$. This completes the proof.
\end{proof}
%【REVISE】 After Proposition \ref{prop:4}, we spend explanations on Main Theorem \ref{mthm:2}. 
Last, we present an interesting example related to this theorem.
%【REVISE0】 explain → present .
For example, in \cite{barden, crowley, nishioka, smale}, a $5$-dimensional closed, simply-connected and smooth manifold $S^{\prime \prime}$ satisfying $H_2(S^{\prime \prime};\mathbb{Z}) \cong \mathbb{Z}/2\mathbb{Z} \oplus \mathbb{Z}/2\mathbb{Z}$ which can be smoothly immersed and embedded into ${\mathbb{R}}^6$ is presented.
%【REVISE0】 In \cite{barden}, \cite{crowley}, \cite{nishioka}, \cite{smale} and so on → For example, in \cite{barden, crowley, nishioka, smale}
%【REVISE】 into ${\mathbb{R}}^6$  → in ${\mathbb{R}}^6$ .
 Let $D^{\prime \prime}$ be a manifold diffeomorphic to $S^{\prime \prime} \times [-1,1]$. We can smoothly embed this into ${\mathbb{R}}^6$ easily. Set the embedding as ${\bar{f}}_N$ in Proposition {\rm \ref{prop:2}} {\rm (}{\rm \ref{prop:2.2}}{\rm )}. Let $m$ be an arbitrary integer greater than $5$. We can construct a map $f$ as the composition of the first map $f_0$ with the second map $i$ in the following two.
 \begin{itemize}
 	\item The product map $f_0:S^{m-5} \times S^{\prime \prime} \rightarrow [-1,1] \times  S^{\prime \prime}$ of a Morse function with exactly two singular points on the ($m-5$)-dimensional unit sphere $S^{m-5}$ with the (manifold of the) target restricted to a small and closed interval containing the image in the interior and the identity map on $S^{\prime \prime}$: the interval is denoted by [-1,1].
 	\item The immersion (embedding) $i:[-1,1] \times S^{\prime \prime} \rightarrow {\mathbb{R}}^6$: (the manifold of) the domain of the immersion (resp. embedding) is regarded as and can be identified with $D^{\prime \prime}$.
 	\end{itemize}
 
 %【REVISE0】 a product → the product
 %【REVISE0】 sphere  with → sphere with
 %【REIVSE0】 the target → the (manifold of the) target
 %【REVISE0】 the domain of the immersion is regarded as and can be identified with $D^{\prime \prime}$ → (the manifold of) the domain of the immersion is regarded as and can be identified with $D^{\prime \prime}$
 %【REVISE】 We use "itemize".
\begin{Cor}
\label{cor:1}
Let $m$ be an arbitrary integer greater than $6$.
Let $M:=S^{m-5} \times S^{\prime \prime}$ where $S^{\prime \prime}$ is a $5$-dimensional closed, simply-connected and smooth manifold $S^{\prime \prime}$, satisfying $H_2(S^{\prime \prime};\mathbb{Z}) \cong \mathbb{Z}/2\mathbb{Z} \oplus \mathbb{Z}/2\mathbb{Z}$, being smoothly immersed {\rm (}embedded{\rm )} into ${\mathbb{R}}^6$, and presented just before. If $m>7$ holds, then $M$ admits no special generic maps into ${\mathbb{R}}^n$ for $1 \leq n \leq 5$. If $m=7$ holds here, then $M$ admits no special generic maps into ${\mathbb{R}}^n$ for $1 \leq n \leq 4$.
%【REVISE_on_Comment3】 We start from "Let $M:=S^{m-5} \times S^{\prime}$ where $S^{\prime \prime}$ is as just before. If" . 
\end{Cor}
%【REVISE0】 domain $M$ → the manifold $M$ of the domain .
\begin{proof}
By the topology of $M:=S^{m-5} \times S^{\prime \prime}$, 
we can apply K\"unneth formula to have that $H^{\ast}(M;\mathbb{Z})$ is isomorphic to the tensor product $H^{\ast}(S^{m-5};\mathbb{Z}) \otimes H^{\ast}(S^{\prime \prime};\mathbb{Z})$ of $H^{\ast}(S^{m-5};\mathbb{Z})$ and $H^{\ast}(S^{\prime \prime};\mathbb{Z})$, defined as the set of all elements represented as the tensor product $u_1 \otimes u_2$ for some $(u_1,u_2) \in H^{\ast}(S^{m-5};\mathbb{Z}) \times H^{\ast}(S^{\prime \prime};\mathbb{Z})$. This element is identified naturally with the cup product of elements identified in a canonical way with $u_1$ and $u_2$. In short, we consider the pullbacks by the projections and consider the cup product.
We have the cup product of an element of $H^{3}(M;\mathbb{Z})$ identified naturally with an element of $H^3(S^{\prime \prime};\mathbb{Z}) \cong \mathbb{Z}/2\mathbb{Z} \oplus \mathbb{Z}/2\mathbb{Z}$ which is not the zero element and an element of $H^{m-5}(M;\mathbb{Z})$ which is identified naturally with an element of $H^{m-5}(S^{m-5};\mathbb{Z})$ regarded as the cohomology dual to a homology class represented by a standard sphere of the domain of a Morse function on an ($m-5$)-dimensional standard sphere: in the latter we can choose two generators of $H_{m-5}(S^{m-5};\mathbb{Z})$. This cup product is not the zero element. This element is also an element of $H^{m-2}(M;\mathbb{Z})$.

 Theorem \ref{thm:5} completes the proof for $m>7$, satisfying the inequality $m-2>5$. 
 
 For $m=7$, suppose that the $7$-dimensional manifold here admits a special generic map into ${\mathbb{R}}^n$ for $1 \leq n \leq 4$. We can find an element of $H^2(M;\mathbb{Z}) \cong \mathbb{Z}$ and an element of $H^3(M;\mathbb{Z}) \cong \mathbb{Z}/2\mathbb{Z} \oplus \mathbb{Z}/2\mathbb{Z}$ whose cup product is not the zero element. This element is also an element of $H^{5}(M;\mathbb{Z})$, satisfying the inequality $5>4$. This contradicts Theorem \ref{thm:5}. This completes the proof for $m=7$.
 
 This completes the proof.   
%【REVISE0】 the dual of →　regarded as the dual to 

\end{proof}
\begin{Rem}
	We can also prove Corollary \ref{cor:1} for the non-existence of special generic maps into ${\mathbb{R}}^n$ for $n=1,2,3,4$ by applying some important ingredients of the proof of Theorem \ref{thm:8}. In fact $H_2(M;\mathbb{Z}) \cong \mathbb{Z}/2\mathbb{Z} \oplus \mathbb{Z}/2\mathbb{Z}$ holds for our manifold $M$ of dimension $m \geq 7$.
\end{Rem}
\begin{Rem}
	We do not know whether in Corollary \ref{cor:1}, the $7$-dimensional manifold $M:=S^2 \times S^{\prime \prime}$ admits a special generic map into ${\mathbb{R}}^5$. We also omit the case where the dimension is $m=6$. Note that in the case $m=6$, $M$ is not simply-connected and that the fundamental group ${\pi}_1(M)$ is isomorphic to $\mathbb{Z}$.
\end{Rem}
The $j$-th {\it Stiefel Whitney class} of a linear bundle is the uniquely defined element of $H^j(X;\mathbb{Z}/2\mathbb{Z})$ of the base space $X$. The $j$-th {\it Pontrjagin class} of a linear bundle is the uniquely defined element of $H^{4j}(X;\mathbb{Z})$ of the base space $X$.
In short, they represent so-called obstructions for the bundles to be trivial. We consider these notions of a smooth manifold as those of the tangent bundle.

The {\it Whitney sum} of two linear bundles over a topological space $X$ is a kind of the direct sum of two linear bundles over a space. More rigorously, by considering the pull-backs, we have a new linear bundle over $X \times X$ in a canonical way. We consider the pullback of the new bundle by the diagonal inclusion from $X$ to $X \times X$. This is the {\it Whitney sum} of the original two bundles. 

As presented in introduction, see \cite{milnorstasheff}.

The following proposition is a very fundamental fact. See the book.

\begin{Prop}
	\label{prop:7}
The $j$-th Stiefel-Whitney classes and the $j$-th Pontrjagin classes of a homotopy sphere are always the zero elements for any positive integer $j$.
Consider arbitrary smooth manioflds whose $j$-th Stiefel-Whitney classes and $j$-th Pontrjagin classes are always the zero elements for any positive integer $j$. Smooth manifolds represented as connected sums of these manifolds and products of these manifolds also enjoy this property.	
\end{Prop}

\begin{Prop}
	\label{prop:8}
	For an $m$-dimensional closed and orientable manifold $M$ admitting a product-organized special generic map $f:M \rightarrow {\mathbb{R}}^n$, the $j$-th Stiefel-Whitney class and the $j$-th Pontrjagin class are the zero elements for any positive integer $j>0$.
\end{Prop}
\begin{proof}

This is due to \cite{eliashberg}. This admits a special generic map into ${\mathbb{R}}^m$. The Whitney sum of the tangent bundle and a trivial linear bundle whose fiber is the $1$-dimensional Euclidean space $\mathbb{R}$ is trivial. This with fundamental properties of Stiefel-Whitney classes and Pontrjagin classes gives the desired fact. This completes the proof.
\end{proof}
\begin{proof}[A proof of Main Theorem \ref{mthm:3}]
We construct a special generic map in Theorem \ref{thm:8} on a closed and simply-connected manifold $M$ of dimension $m=9$.
This $m$-dimensional manifold is simply-connected. \\
\ \\
\underline{Claim 1} \ For any element of $H^{j_1}(M;\mathbb{Z})$ and any element of $H^{j_2}(M;\mathbb{Z})$, the cup product is the zero element where $j_1,j_2>0$ and $0<j_1+j_2<m=9$ are assumed. \\
\ \\
Since $1 \leq j_1,j_2 \leq 2$, the elements are always the zero elements, we may regard $j_1,j_2 \geq 3$. $H^{m-1}(M;\mathbb{Z})=H^8(M;\mathbb{Z})$ is the trivial group. It is sufficient to consider the case $(j_1,j_2)=(3,3),(3,4),(4,3)$. However, the cup products are always the zero elements by Theorem \ref{thm:5}. This completes the proof of Claim 1. \\
\ \\
\underline{Claim 2} If we restrict the $H^{\ast}(M;\mathbb{Z})$ to the set of all elements which are the zero elements or which  are elements of infinite order, then we have a subalgebra and this is isomorphic to $H^{\ast}(S^4 \times S^{m-4};\mathbb{Z})=H^{\ast}(S^4 \times S^5;\mathbb{Z})$. \\
\ \\
This is a corollary to Claim 1 together with Poincar\'e duality.  \\ 
\ \\
We prepare $2l$ copies of a special generic map before and finitely many copies of a special generic map into ${\mathbb{R}}^5$ in Example \ref{ex:1} on $S^j \times S^{m-j}=S^j \times S^{9-j}$ whose restriction to the singular set is an embedding and whose image is diffeomorphic to $S^j \times D^{5-j}$ where $j=2,3,4$.

By an elementary construction, used in \cite{saeki2} for example, we can easily construct a special generic map into ${\mathbb{R}}^5$ on a manifold $M_1$ of dimension $m=9$ represented as a connected sum of these manifolds. We can also have a special generic map whose restriction to the singular set is an embedding.

Instead, we prepare a special generic map into ${\mathbb{R}}^n$ in Theorem \ref{thm:8} as a product-organized one, which is introduced in Definition \ref{def:1} as an important special generic map in Theorem \ref{thm:7}. We have a special generic map there into ${\mathbb{R}}^n$ for $6 \leq n \leq m$ ($n=6,7,8,9$). We prepare $2l$ copies of the special generic map as before and finitely many copies of a special generic map into ${\mathbb{R}}^n$ in Example \ref{ex:1} on $S^j \times S^{m-j}=S^j \times S^{9-j}$ whose restriction to the singular set is an embedding and whose image is diffeomorphic to $S^j \times D^{n-j}$ where $j=2,3,4$.

%【REVISE0】 the resulting → a resulting .

We take $M_2$ as a manifold represented as a connected sum of $l$ copies of a manifold represented as $S^{\prime \prime} \times S^{m-5}=S^{\prime \prime} \times S^{4}$ where $S^{\prime \prime}$ is a manifold used in Corollary \ref{cor:1} and finitely many copies of $S^2 \times S^{m-2}=S^2 \times S^7$, $S^3 \times S^{m-3}=S^3 \times S^6$, and $S^4 \times S^{m-4}=S^4 \times S^5$. We construct a special generic map into ${\mathbb{R}}^n$ for $6 \leq n \leq m$ ($n=6,7,8,9$).

For copies of $S^{\prime \prime} \times S^{m-5}=S^{\prime \prime} \times S^{4}$, we can construct special generic maps into ${\mathbb{R}}^n$. 
We consider the product map of a canonical projection of the unit sphere into ${\mathbb{R}}^{n-5}$ and the identity map on $S^{\prime \prime}$ and we can embed the manifold of the target into ${\mathbb{R}}^n$ to obtain a special generic map for $6 \leq n \leq m$ ($n=6,7,8,9$).

We can also do the construction for the desired manifolds represented as connected sums as before.

 By taking suitable connected sums, we can see that the pair of the manifolds enjoys the properties (\ref{mthm:3.1}) and (\ref{mthm:3.2}), and due to the non-existence, shown by virtue of Theorem \ref{thm:5}, (the proof of) Theorem \ref{thm:8}, or (that of) Corollary \ref{cor:1}, enjoys the property (\ref{mthm:3.3}). Propositions \ref{prop:7} and \ref{prop:8} yield the property (\ref{mthm:3.4}) and completes the proof.

\end{proof}
\begin{MThm}
	%【REVISE0】 We have deleted "." after \ref{thm:10}.
	\label{mthm:4}
	Let $l>0$ be a positive integer. Let $\{G_j\}_{j=0}^8$ be a sequence of finitely generated commutative groups of length $9$ satisfying the following conditions.
	%【REVISE0】 satisfying → enjoying .
	\begin{itemize}
		\item $G_j$ and $G_{8-j}$ are isomorphic for $0 \leq j \leq 8$. $G_0$ is isomorphic to $\mathbb{Z}$ and $G_1$ is trivial.
		\item The torsion subgroup of $G_2$ is isomorphic to a group represented as a direct sum of $2l$ copies of the group of order $2$.
		\item For $j \neq 2,5$, $G_j$ is free. The rank of $G_4$ is at least $4l$ and even.
	\end{itemize}
	%【REVISE0】 We use "itemize" .
	Then there exists a pair $(M_1,M_2)$ of $8$-dimensional closed and simply-connected manifolds enjoying the following properties.
	%【REVISE0】 satisfying → enjoying .
	\begin{enumerate}
		\item The $j$-th integral homology groups of $M_1$ and $M_2$ are isomorphic to $G_j$ for $0 \leq j \leq m$.
		\item The sets consisting of all elements whose orders are infinite and the zero elements of the integral cohomology rings of $M_1$ and $M_2$ are subalgebras isomorphic to the integral cohomology ring of a $8$-dimensional manifold, which is represented as a connected sum of products of standard spheres and presented in Example \ref{ex:1}.
			\item Let $i=1,2$. $M_i$ admits a special generic map into ${\mathbb{R}}^{n}$ for $i+4 \leq n \leq 8$ whereas it admits no special generic maps into ${\mathbb{R}}^{n}$ for $1 \leq n \leq i+3$. Furthermore, we can construct the special generic map on $M_i$ so that the restriction to the singular set is an embedding.
		\item The $j$-th Stiefel-Whitney classes and the $j$-th Pontryagin classes of $M_1$ and $M_2$ are the zero elements for any positive integer $j$. 
	
	\end{enumerate}
\end{MThm}
%【REVISE】 "Example 2" in the previous version has been deleted.
\begin{proof}[A proof of Main Theorem \ref{mthm:4}]
	We can prove similarly by replacing $m=9$ by $m=8$ in the proof of Main Theorem \ref{mthm:3} except Claims 1 and 2. We revise them. Let $f:M \rightarrow {\mathbb{R}}^n$ be a special generic map on a closed and simply-connected manifold of dimension $m=8$ in Theorem \ref{thm:8}. \\
	\ \\	
\underline{Claim 3} \ For any element of $H^{j_1}(M;\mathbb{Z})$ and any element of $H^{j_2}(M;\mathbb{Z})$, the cup product is the zero element where $j_1,j_2>0$ and $0<j_1+j_2<m=8$ are assumed. \\
\ \\
Since $1 \leq j_1,j_2 \leq 2$, the elements are always the zero elements, we may regard $j_1,j_2 \geq 3$. $H^{m-1}(M;\mathbb{Z})=H^7(M;\mathbb{Z})$ is the trivial group. It is sufficient to consider the case $(j_1,j_2)=(3,3)$. However, the cup products are always the zero elements by Theorem \ref{thm:5}. This completes the proof of Claim 3. \\
\ \\
\underline{Claim 4} If we restrict $H^{\ast}(M;\mathbb{Z})$ to the set of all elements which are the zero elements or which are elements of infinite order, then we have a subalgebra and this is isomorphic to $H^{\ast}(S^4 \times S^{m-4};\mathbb{Z})=H^{\ast}(S^4 \times S^{4};\mathbb{Z})$. \\
\ \\
For a topological space $X$, let $\chi(X)$ denote its Euler number.
For a special generic map $f:M \rightarrow {\mathbb{R}}^5$ here, 
\begin{eqnarray*}
	\chi(M)&=&\chi(\partial W_f)\chi(D^{m-4})+\chi(W_f)\chi(S^{m-5}) \\
	&=&\chi(\partial W_f)\chi(D^{4})+\chi(W_f)\chi(S^{3}) \\
	&=&(2+2) \times 1+0\\&=&4.
\end{eqnarray*}
Remember that we have a special generic map 
$f$ as in Theorem \ref{thm:8}. $\partial W_f$ is, as discussed in Propositions \ref{prop:4} and \ref{prop:5} and the proofs, diffeomorohic to the disjoint union of $B$ in Propostion \ref{prop:4} and the $4$-dimensional unit sphere and that the Euler number is calculated as $(1+1)+2=2+2=4$.

The ranks of $H_j(M;\mathbb{Z}) \cong H_j(W_f;\mathbb{Z})$ and $H^j(M;\mathbb{Z}) \cong H^j(W_f;\mathbb{Z})$ are $0$ for $j=1,2,3$ by Proposition \ref{prop:3} (\ref{prop:3.5}). 
The ranks of $H_j(M;\mathbb{Z})$ and $H^j(M;\mathbb{Z})$ are shown to be not $0$ if and only if $j=0,4,8$. The ranks of $H_4(M;\mathbb{Z})$ and $H^4(M;\mathbb{Z})$ are shown to be $2$ by the property that $M$ is connected and orientable, showing that the ranks of $H^0(M;\mathbb{Z})$ and $H^8(M;\mathbb{Z})$ are $1$, and the Euler number of $M$. 

We apply some arguments in section 3 of \cite{saeki2} to determine the structure of the subalgebra and complete the proof of Claim 4.

We can smoothly embed a $1$-dimensional compact and connected manifold $L$ diffeomorphic to a closed interval in $W_f$ in such a way that the following two are enjoyed.
\begin{itemize}
\item The interior ${\rm Int}\ L$ is embedded in the interior ${\rm Int}\ W_f$.
\item The boundary $\partial L$ is embedded in the boundary $\partial W_f$. The boundary consists of exactly two points and they are embedded in different connected compoennts of the boundary $\partial W_f$.
\end{itemize}
Moreover, we can construct the composition of the restriction of $q_f$ to ${q_f}^{-1}(L)$ with some embedding into $\mathbb{R}$ and we may regard this as a Morse function on a $4$-dimensional standard sphere in $M$. An element of $H_4(M;\mathbb{Z})$ represented by this is regarded as the Poincar\'e dual to the cohomology dual to an element represented by some connected component $C_{S(f)}$ of the singular set of the special generic map $f$, which is a $4$-dimensional closed and connected manifold smoothly embedded into $M$. Note that we can define a basis of $H_4(M;\mathbb{Z})$ by taking these two elements, which are mapped to a generator of $H_4(W_f;\mathbb{Z}) \cong \mathbb{Z}$ and the zero element of $H_4(W_f;\mathbb{Z})$ by the homomorphism ${q_f}_{\ast}:H_4(M;\mathbb{Z}) \rightarrow H_4(W_f;\mathbb{Z})$, respectively.

We show that each of these two $4$-dimensional closed and connected submanifolds ${q_f}^{-1}(L)$ and $C_{S(f)}$ in $M$ is smoothly isotoped to another $4$-dimensional closed and connected submanifold apart from the original one. 

We can easily see that we can smoothly isotope $L$ in $W_f$ to another $1$-dimensional smooth submanifold manifold diffeomorphic to a closed interval apart from the original one. 
We can smoothly isotope ${q_f}^{-1}(L)$ to another $4$-dimensional sphere apart from the original sphere.

We can have the map $f$ such that the "linear bundle over $N(\partial W_f)$ in Propostion \ref{prop:2}" is trivial by Proposition \ref{prop:2} (\ref{prop:2.2}). The triviality of this bundle means that We can smoothly isotope the connected component $C_{S(f)}$ of the singular set to another $4$-dimensional closed and connected manifold apart from the original $4$-dimensional submanifold.
 
Poincar\'e duality theorem for $M$ completes the proof of Claim 4. \\

We can prove all similarly by setting $m=8$ in the proof of Main Theorem \ref{mthm:3} instead. This completes the proof.
\end{proof}

\section{Acknowledgment and on data availability.}
The present work and the author is supported by JSPS KAKENHI Grant Number JP17H06128 "Innovative research of geometric topology and singularities of differentiable mappings"
(https://kaken.nii.ac.jp/en/grant/KAKENHI-PROJECT-17H06128/) as a member of the project: Principal investigator is Osamu Saeki.
%【REVISE】 : Principal investigator is Osamu Saeki .
 Independently, this work has been also supported by "The Sasakawa Scientific Research Grant" (2020-2002 : https://www.jss.or.jp/ikusei/sasakawa/).

The author declares that data supporting the present study are available within the present paper.

\end{document}